\documentclass[12pt,a4paper,oneside]{amsart}
\usepackage[utf8]{inputenc}
\usepackage{amsmath,amscd,hyperref, amsfonts, amssymb, amsthm,}
\usepackage{setspace,kantlipsum}
\usepackage{tikz-cd}
\usepackage{mathrsfs}
\usepackage{textcomp}
\usepackage[margin=1in]{geometry}
\usepackage{colonequals}
\usepackage{mathtools}
\usepackage{stackrel}
\usepackage{amssymb}
\usepackage[hyphenbreaks]{breakurl}

\theoremstyle{plain} 
\newtheorem{thm}{Theorem}[section]
\newtheorem{lemma}[thm]{Lemma}
\newtheorem{prop}[thm]{Proposition}
\newtheorem{corollary}[thm]{Corollary}

\theoremstyle{definition}

\newtheorem{definition}[thm]{Definition}

\newtheorem{remark}[thm]{Remark}

 \theoremstyle{plain} % just in case the style had changed
\newcommand{\thistheoremname}{}
\newtheorem{genericthm}[thm]{\thistheoremname}

\newtheorem*{genericthm*}{\thistheoremname}
\newenvironment{namedthm*}[1]
  {\renewcommand{\thistheoremname}{#1}%
   \begin{genericthm*}}
  {\end{genericthm*}}

\newcommand{\C}{\mathbb{C}}

\newcommand{\Q}{\mathbb{Q}}
\newcommand{\N}{\mathbb{N}}
\newcommand{\Z}{\mathbb{Z}}

\newcommand{\cM}{\mathcal{M}}
\newcommand{\cN}{\mathcal{N}}
\newcommand{\cO}{\mathcal{O}}

\newcommand{\cS}{\mathcal{S}}

\newcommand{\sD}{\mathscr{D}}

\renewcommand{\d}{\partial}

\newcommand{\GL}{\mathrm{GL}}
\newcommand{\gr}{\mathrm{gr}}
\newcommand{\IC}{\mathrm{IC}}

\newcommand{\Sym}{\mathrm{Sym}}

\newcommand{\lie}{\mathfrak{g}}

\newcommand{\op}{\operatorname}

\DeclareMathOperator{\Ker}{Ker}

\DeclareMathOperator{\supp}{supp}
\title{Filtrations of $\sD$-modules along semi-invariant functions}
\author{Andr\'as C. L\H{o}rincz and Ruijie Yang}

\begin{document}

\begin{abstract}
 Given a smooth algebraic variety $X$ with an action of a connected reductive linear algebraic group $G$, and an equivariant $\sD$-module $\cM$, we study the $G$-decompositions of the associated $V$-, Hodge, and weight filtrations.

If $\cM$ is the localization of a $\sD$-module $\cS$ underlying a pure twistor $\sD$-module (e.g. when $\cS$ is simple) along a semi-invariant function, we determine the weight level of any element in an irreducible isotypic component of $\cM$ in terms of multiplicities of roots of $b$-functions. If $\cS$ underlies a pure Hodge module, we show that the Hodge level is governed by the degrees of another class of polynomials, also expressible in terms of $b$-functions. 
 
 As an application, if $X$ is an affine spherical variety, we describe these filtrations representation-theoretically in terms of roots of $b$-functions, and compute all higher multiplier and Hodge ideals associated with semi-invariant functions. Explicit examples include the spaces of general, skew-symmetric, and symmetric matrices, as well as the Freudenthal cubic on the fundamental representation of $E_6$.

\end{abstract}

\maketitle
%\tableofcontents

\section{Introduction}

The main goal of this paper is to compute the $V$-filtration of Kashiwara and Malgrange for equivariant $\sD$-modules along semi-invariant functions with respect to a reductive linear algebraic group, using representation-theoretic techniques. Among the applications, we characterize the weight level and Hodge level of elements in an irreducible isotypic component of the localization of $\sD$-modules along such semi-invariant functions.

In the theory of $\sD$-modules, the $V$-filtration determines the nearby and vanishing cycles and plays an essential role in the theory of mixed Hodge/twistor modules \cite{Saito88,MochizukitwistorDmodule}. On the other hand, beginning with the work of Budur and Saito \cite{BS05}, it has been observed that several key invariants in birational geometry - such as the log canonical threshold and multiplier ideals - are shadows of the $V$-filtration on the structure sheaf. In recent years, this link has been explored intensively. For example, by leveraging the Hodge filtration on the $V$-filtration, \emph{Hodge ideals} \cite{MP16} and \emph{higher multiplier ideals} \cite{SY23} (see also \cite{Saito16}) have been introduced as refinements of multiplier ideals.

Despite its significance, the explicit computation of the $V$-filtration remains challenging and has been carried out only for a few classes of polynomials:  semi-quasi-homogeneous polynomials \cite{Saito09}, normal crossing divisors \cite{Saito88} and Newton-nondegenerate polynomials \cite{Saito94}. In subsequent work, the corresponding higher multiplier and Hodge ideals were also computed, see \cite{SY23,MP16,Zhang,Saito16,JKSY}. For related work on free or parametrically prime divisors, see also \cite{freedivisor,BD24,dakin}.

The starting point of this paper is the observation that, in the presence of a group action, one can explicitly compute the isotypic components of the $V$-filtrations along semi-invariant functions - provided they correspond to irreducible isotypic components of equivariant $\sD$-modules.  In particular, this yields a large new class of examples of hypersurfaces that often have \emph{non-isolated} singularities, including semi-invariant functions on affine spherical varieties.

\subsection{Filtrations of equivariant $\sD$-modules and irreducible isotypic components}\label{sec:1.1}
Suppose we have an affine space $X$ with an action by a connected reductive linear algebraic group $G$. We will focus on a left (strongly) \emph{$G$-equivariant} $\sD_X$-module $\cM$, along a semi-invariant function $f$ of weight $\sigma$  with respect to $G$. For technical reasons, we require  $\cM$ to be regular holonomic with quasi-unipotent monodromy along $f$ (which is automatic when $X$ has only finitely many $G$-orbits; see \cite[Theorem 11.6.1]{HTT}).

Next, we want to look at the graph embedding of $f$, denoted $\iota:X\to X\times \C$, and consider its direct image $\iota_{+}\cM$. By the semi-invariance of $f$, we have a natural diagonal $G$-action on $X\times \C$ such that the map $\iota$ is $G$-equivariant, and the $V$-filtration $V^{\alpha}\iota_{+}\cM$ along $\{t=0\}$ is (weakly) $G$-equivariant for all $\alpha$, where $t$ is the coordinate of $\C$ (see Lemma \ref{lemma: G and s compatible on V filtration}). For simplicity, we work with $\mathcal{M} = \cS_f$, the localization of a \emph{simple} equivariant $\sD$-module $\cS$ along $f$, supported on $X$. Then there is an isomorphism of $\sD_X\langle s,t\rangle$-modules
\begin{equation}\label{eqn: Malgrange isomorphism intro}
\iota_{+} \mathcal{M} \xrightarrow{\sim} \mathcal{M}[s] f^s,
\end{equation}
where $s = -\partial_t t$ and $t$ acts by $s\mapsto s+1$ (cf. \cite{Malgrange,MPVfiltration}). 

Let $\Lambda$ be the set of integral dominant weights of $G$, parameterizing the finite-dimensional irreducible $G$-modules. For $\lambda\in \Lambda$, let $\cM_{\lambda}$ denote the associated $G$-isotypic component of $\cM$. From now on, we fix a $\lambda \in \Lambda$ such that 
\[ \textrm{$\cM_\lambda$ is an \emph{irreducible} $G$-module}.\]
\noindent Under this assumption, $f$ must be homogeneous \cite[Proposition 4.3]{saki} and we denote $d\colonequals \deg f$. In Proposition \ref{prop:bfunirred}, we show that there exist a unique semi-invariant differential operator $\partial f$ with constant coefficients of weight $-\sigma$, and a degree $d$ polynomial $b_{\lambda}(s)$, such that for any $0\neq m\in \cM_{\lambda}$,
\begin{equation}\label{eqn: blambda} \partial f \cdot m f^{s+1} = b_\lambda(s) \cdot m f^s,\end{equation}
and $b_m(s)=b_{\lambda}(s) \colonequals\prod_{i=1}^{d} (s + \lambda_i)$ with $\lambda_i\in \Q$. This type of equation originates from the idea of M. Sato for computing $b$-functions of
semi-invariants on prehomogeneous spaces, and is further developed in recent work \cite{holoprehom}. See Section \ref{sec:moreb} for additional explicit results on such $b_{\lambda}(s)$.

To describe the $V$-filtration, we define a new class of polynomials: for $m\in \cM$ and $\alpha \in \Q$, the $p$-\emph{function} $p_{m,\alpha}(s)\in \C[s]$  is the monic polynomial of smallest degree such that $p_{m,\alpha}(s)mf^s\in V^{\alpha}\iota_{+}\cM$. Due to $G$-equivariance of the $V$-filtration, any two nonzero elements in $\cM_{\lambda}$ yield the same polynomial, denoted by $p_{\lambda,\alpha}(s)$.  Set 
\begin{equation}\label{eqn: [s+a]}
[s + a]_k\colonequals \begin{cases}\prod_{i=0}^{k-1} (s + a + i), \quad &\text{if } k \geq 1,\\
1, \quad &\text{otherwise}.
\end{cases} \end{equation}
\begin{thm}\label{thm: V filtration general G}
The isomorphism \eqref{eqn: Malgrange isomorphism intro} induces a $(\C[s],G)$-isomorphism:
\[ (V^{\alpha}\iota_{+}\cM)_{\lambda}\cong (p_{\lambda,\alpha}(s))\cdot \cM_{\lambda}f^s,\]
where $(p(s))$ denotes the ideal in $\C[s]$ generated by $p(s)$, and
\begin{equation}\label{eqn: plambdaalpha in general}
p_{\lambda,\alpha}(s) = \prod_{i=1}^{d} [s + \lambda_i]_{\lceil \alpha - \lambda_i\rceil}.
\end{equation}
\end{thm}
Since $\cS$ is simple, it underlies a \emph{pure twistor $\sD$-module} on $X$ \cite{Mochizuki}, say of weight $q$. Then $\cM\cdot f^{-\alpha}$ underlies a mixed twistor $\sD$-module for $\alpha\in \Q$ \cite{MochizukitwistorDmodule}, and hence carries a weight filtration starting at weight $q$. If $\cS$ underlies a polarized Hodge module (e.g. $\cS=\cO_X$) \cite{Saito88}, then $\cS$ also underlies a pure twistor module, and the corresponding weight filtrations on $\cM$ coincide up to a shift.  It turns out that the weight filtration is controlled by the multiplicity of $s=-\alpha$ as a root the $b$-function of $p_{m,\alpha}(s)mf^s$ for $m\in \cM$, denoted by $\nu_{m,\alpha}$. Using the expression for $p_{\lambda,\alpha}(s)$ \eqref{eqn: plambdaalpha in general} and Theorem \ref{thm:bfun}, for a nonzero $m\in \cM_{\lambda}$ we obtain
\begin{equation}\label{eqn: numalpha}\nu_{\lambda,\alpha}\colonequals\nu_{m,\alpha} = \# \{ i \mid \alpha - \lambda_i \in \mathbb{Z}_{\geq 0} \}.\end{equation} 
\begin{thm}\label{thm:weightisot}
    For any non-zero $m \in \cM_\lambda$, $\ell\in \N$ and $\alpha\in \Q$, we have:
  \[ m \cdot f^{-\alpha} \, \in W_{q+\ell}(\cM\cdot f^{-\alpha})  \Longleftrightarrow \nu_{m,\alpha}\leq \ell. \]
\end{thm}
\begin{corollary}\label{cor: ValphaiotaS}
    Let $r_\lambda$ be the largest integer root of $(s+1)b_\lambda(s)$, then 
\begin{equation}\label{eqn: V0Slambda}
(V^\alpha \iota_{+} \mathcal{S})_\lambda= \left(\mathrm{lcm}\{[s-r_{\lambda}]_{r_{\lambda}+1}, p_{\lambda,\alpha}(s)\}\right) \cdot \cM_\lambda f^s.
\end{equation}
\end{corollary}
If $\cS$ underlies a complex pure Hodge module, then $\cM\cdot f^{-\alpha}$ underlies a complex mixed Hodge module \cite{MHMproject} and is equipped with a (functorial) Hodge filtration. Our next result shows that this is controlled by the \emph{degree} of $p_{m,\alpha}(s)$, which is equal to 
\[\sum_{i=1}^d \max\{\lceil\alpha-\lambda_i \rceil , 0 \} = \sum_{\beta < \alpha} \nu_{m, \beta}.\]
For simplicity, we state here only for $\cS=\cO_X$; for the general $\cS$ see Theorem \ref{thm:FS}.
\begin{thm}\label{thm:hodgeisot}
Let $\cM=(\cO_X)_f$. For a nonzero $m\in \cM_\lambda$, $k\in \N$ and $\alpha\in  \Q$, we have
    \[m \cdot f^{-\alpha} \, \in F_{k} (\cM\cdot f^{-\alpha}) \Longleftrightarrow 
\deg p_{m,\alpha}(s)\leq k.\]
\end{thm}

The proof of Theorem \ref{thm:hodgeisot} uses the following technical result (for the general $\cS$ see Proposition \ref{prop:hodgeVgen}). Using the evaluation map $\mathrm{ev}_{s=-\alpha}:V^{\alpha}\iota_{+}\cM\to \cM\cdot f^{-\alpha}$ from \eqref{eqn: evaluation map} and \cite[Corollary 1.2]{DY24}, one has \begin{equation}\label{eqn: Hodge filtration on M via V filtration}
F_k(\cM\cdot f^{-\alpha})= \mathrm{ev}_{s=-\alpha}(V^{\alpha}\iota_{+}\cM\cap F_{k+1}\iota_{+}\cM), \quad \forall \alpha\geq 0,\end{equation}
where $F_{k+1}\iota_{+}\cM = \sum_{i+j\leq k} F_i\cM\otimes \d_t^j$. 
\begin{thm}\label{thm: V cap F any alpha general G}
Let $\cM=(\mathcal{O}_X)_f$ and $k\in \N$. Then
\begin{align*}
    (V^{\alpha}\iota_{+}\cM\cap F_{k+1}\iota_{+}\cM)_{\lambda} =\begin{cases}\left(\mathrm{lcm}\{[s-\ell_{k}+1]_{\ell_{k}},p_{\lambda,\alpha}(s)\}\right)_{\leq k}\cdot \cM_\lambda f^s, &\textrm{if $\alpha\in \Q$}, \\
    (p_{\lambda,\alpha}(s))_{\leq k}\cdot \cM_\lambda f^s, &\textrm{if $\alpha\in \Q_{>0}$}. \end{cases}
\end{align*}
Here $I_{\leq k}$ denotes the subspace of elements of degree $\leq k$ in an ideal $I$, and $\ell_{k}$ is the minimal $\ell\geq 0$ such that $(F_{k-\ell}\cM)_{\lambda+\ell\sigma}\neq 0$, or equivalently, such that $\deg p_{\lambda,-\ell}(s)\leq k-\ell$ (if no such $\ell$ exists, set $[s-\ell_{k}+1]_{\ell_{k}}=0$).
\end{thm}
This technical result also allows us to describe the Hodge filtration on the nearby and vanishing cycles mixed Hodge modules, which is defined by
\[ F_k\gr^{\alpha}_V\iota_{+}\cM=\frac{V^{\alpha}\iota_{+}\cM\cap F_k\iota_{+}\cM}{V^{>\alpha}\iota_{+}\cM\cap F_k\iota_{+}\cM}.\]
Note that by Theorem \ref{thm: V filtration general G}, one has
\[(V^{>\alpha}\iota_{+}\cM)_{\lambda}=\left((s+\alpha)^{\nu_{\lambda,\alpha}}p_{\lambda,\alpha}(s)\right)\cdot \cM_{\lambda}f^s,\]
(see also \cite[Lemma 2.6]{DLY}), therefore we have the following.
\begin{corollary}\label{cor: Hodge filtration on nearby cycle}
Let $\cM=(\cO_X)_f$, then
\[ \left(F_k\gr^{\alpha}_V\iota_{+}\cM\right)_{\lambda}=\frac{(p_{\lambda,\alpha}(s))_{\leq k}}{((s+\alpha)^{\nu_{\lambda,\alpha}}p_{\lambda,\alpha}(s))_{\leq k}}\otimes \cM_{\lambda}f^s.\]
\end{corollary}

As an application, we compare the $\lambda$-isotypic components of Hodge and higher multiplier ideals, defined as:
\begin{equation}\label{eqn: definition of Hodge ideal}
I_k(\alpha D)\otimes \cO_X(kD)\cdot f^{-\alpha}=F_k((\cO_X)_f\cdot f^{-\alpha}), \quad \forall \alpha>0,\end{equation}
\begin{equation}\label{eqn: definition of higher multiplier ideals}
\tilde{I}_k(\alpha D)\otimes \cO_X(kD)= \frac{V^{\alpha}\iota_{+}\cO_X\cap F_{k+1}\iota_{+}\cO_X}{V^{\alpha}\iota_{+}\cO_X\cap F_{k}\iota_{+}\cO_X}, \quad \forall \alpha\in \Q.\end{equation}
Since $\tilde{I}_k(\alpha D)=\tilde{I}_{k-1}((\alpha+1)D)$ for $\alpha\leq 0$ \cite[Proposition 3.7]{SY23}, we may assume $\alpha>0$.
\begin{corollary}\label{corollary: isotypic component of Hodge and higher}
Assume $((\cO_X)_f)_{\lambda}$ is irreducible, then $\tilde{I}_k(\alpha D)_{\lambda}=I_k(\alpha D)_{\lambda}$ for $\alpha>0$.
\end{corollary}

Our arguments crucially rely on the fact that the evaluation map \eqref{eqn: evaluation map} is well-behaved under the $G$-action. Namely, the isotypic components of lifts of isotypic elements under the evaluation map are pure tensors (rather than mixed tensors), by Theorem \ref{thm: V filtration general G}.  This is where the equivariance and multiplicity-freeness assumptions are pivotal.

\subsection{Filtrations on affine spherical varieties}
Next, we focus on the case when $X$ is a (smooth) affine spherical variety. By a classical argument based on Luna's slice theorem, our study of filtrations on equivariant $\sD$-modules can readily be reduced to the case where $X$ is a \emph{multiplicity-free space} - that is, a representation of $G$ such that $X$ has a dense orbit (equivalently, finitely many orbits) under the action of a Borel subgroup of $G$. Such spaces have received considerable interest, and are classified by the works of  Kac \cite{Kac}, Brion \cite{Brion}, Benson-Ratcliff \cite{class1}, Leahy \cite{class2}. They are important due to Roger Howe's philosophy that many results in invariant theory can be traced back to multiplicity-free spaces. An example of this that is particularly relevant to our paper is the Capelli identity \cite{howeumeda}. Another work displaying some unique features of multiplicity-free spaces from the point of view of $\sD$-modules is \cite{LW19}.

Let $f$ be a semi-invariant function with respect to $G$ of weight $\sigma$, $\cS$ a simple equivariant $\sD$-module whose support is not contained in $f^{-1}(0)$, and set $\mathcal{M} = \cS_f$,  its localization along $f$. The multiplicity-free assumption on $X$ implies that $\mathcal{M}$ is also multiplicity-free - that is, \emph{each} non-zero isotypic component $\cM_\lambda$ is irreducible. In this setup, the results from Section \ref{sec:1.1} hold for each isotypic component \emph{without} requiring $\mathrm{supp}(\cS)=X$. 

Consider the decomposition 
\[ \mathcal{M} = \bigoplus_{\lambda \in \Lambda(\mathcal{M})} U_\lambda,\]
where $\Lambda(\mathcal{M})$ is the set of integral dominant weights indexing the finite-dimensional irreducible $G$-modules, and $U_\lambda=\cM_\lambda$. Let $b_{\lambda}(s)=\prod (s+\lambda_i)$, as in Proposition \ref{prop:bfunirred}.

\begin{thm}\label{thm: V filtration on multiplicity free space via representation data}
We have $\deg b_\lambda(s) = \deg f$. For any $\alpha\in \Q$, there is a $(\mathbb{C}[s], G)$-decomposition:
\begin{equation*}
V^\alpha \iota_{+} \mathcal{M} = \bigoplus_{\lambda \in \Lambda(\mathcal{M})} (p_{\lambda,\alpha}(s)) \otimes \cM_\lambda f^s,
\end{equation*}
where $p_{\lambda,\alpha}(s) = \prod_{i=1}^d [s + \lambda_i]_{\lceil \alpha - \lambda_i\rceil}$. 
\end{thm}
Set $\gr^\alpha_V \iota_{+} \mathcal{M}\colonequals V^{\alpha}\iota_{+}\cM/V^{>\alpha}\iota_{+}\cM$; if $\alpha\in [0,1]$, they are nearby and vanishing cycles $\sD$-modules of $\cM$. Consider the nilpotent order $N=s+\alpha$ on $\gr^\alpha_V \iota_{+} \mathcal{M}$ and denote by $W(N)_{\bullet}$ the associated monodromy weight filtration. Denote by $\lie$ the Lie algebra of $G$.

\begin{corollary}\label{corollary: G decomposition of grV}
For any $\alpha\in \Q,\ell\in \Z$, there are $(\mathbb{C}[s], G)$-module decompositions
\[ 
\gr^\alpha_V \iota_{+} \mathcal{M} =\bigoplus_{\lambda \in \Lambda(\cM)} \frac{\C[s]}{(s+\alpha)^{\nu_{\lambda,\alpha}}} \otimes_{\C} \cM_{\lambda} f^s,
\]
\begin{align}\label{eqn: G decomposition of grV}
W(N)_\ell \gr^\alpha_V \iota_{+} \mathcal{M} &=\bigoplus_{\lambda \in \Lambda(\cM), \, \nu_{\lambda,\alpha}>-\ell} \frac{\C[s]}{(s+\alpha)^{\min\{\lfloor \frac{\nu_{\lambda,\alpha}+\ell+1}{2} \rfloor,  \nu_{\lambda,\alpha}\}}} \otimes_{\C} \cM_{\lambda} f^s,
\end{align}
where $\nu_{\lambda,\alpha}$ is the quantity from \eqref{eqn: numalpha}. Moreover, $\gr^{W(N)}_{\ell}\gr^{\alpha}_V\iota_{+}\cM$ is a \,$-\alpha\sigma$-twisted equivariant\footnote{The twisted equivariance holds more generally; see Proposition \ref{prop: quotient of weight are equivariant}.} semisimple $\sD_X$-module, with a multiplicity-free $\mathfrak{g}$-decomposition:
\begin{equation}\label{eqn: decomposition of grWgrV}\gr_\ell^{W(N)}\gr_V^\alpha \iota_+ \cM = \, \bigoplus_{\lambda \in \Lambda(\cM), \, \nu_{\lambda,\alpha}+\ell \text{ odd}, \, \nu_{\lambda,\alpha}>|\ell|} \,\, U_{\lambda-\alpha\sigma}.\end{equation}
\end{corollary}

\begin{corollary}\label{thm: Gdecomposition of weight filtration}
Assume $\cS$ underlies a pure twistor $\sD$-module of weight $q$. Then, for $\alpha\in\Q,\ell\in \Z_{\geq 0}$, there is a $\lie$-decomposition:
\begin{equation}\label{eqn: G decomposition of Well M}
W_{q+\ell}(\mathcal{M}\cdot f^{-\alpha}) =\bigoplus_{\nu_{\lambda,\alpha} \leq \ell } \cM_{\lambda}\cdot f^{-\alpha} = \bigoplus_{\nu_{\lambda,\alpha} \leq \ell } U_{\lambda-\alpha \sigma}.
\end{equation}
In particular, $W_{q+d}(\mathcal{M}\cdot f^{-\alpha})=\mathcal{M}\cdot f^{-\alpha}$.
\end{corollary}

Since $\gr_{q+\ell}^W (\cM \cdot f^{-\alpha})$ is a semi-simple (and often simple) $\sD$-module, this also yields a method for determining the $G$-decompositions of the composition factors of $\cM$ (cf. \cite{characters}) -- more generally, see Proposition \ref{prop:simples} in this vain.

\begin{corollary}\label{cor: G decomposition of Hodge filtration}
Let $\cM=(\mathcal{O}_X)_f$. For $\alpha\in \Q$, there is a $\lie$-decomposition: 
\[
F_k (\mathcal{M} \cdot f^{-\alpha}) =\bigoplus_{\deg p_{\lambda,\alpha}(s) \leq k} \cM_{\lambda}\cdot f^{-\alpha}  = \bigoplus_{\deg p_{\lambda,\alpha}(s) \leq k} U_{\lambda-\alpha\sigma}.
\]
Furthermore, when $\alpha=0$, this is a $G$-decomposition.
\end{corollary}

In turn, we obtain the $G$-decomposition of \emph{all} Hodge and higher multiplier ideals.
\begin{corollary}\label{corollary: G decomposition of Hodge ideals}
For $\alpha\in \Q_{>0}$, there is a $G$-decomposition:
\begin{equation}\label{eqn: G decomposition of Hodge ideals}
\tilde{I}_k(\alpha D)=\bigoplus_{\deg p_{\lambda,\alpha+k}(s)\leq k} U_\lambda=I_k(\alpha D).\end{equation}
Furthermore, $\tilde{I}_k(\alpha D)=I_k(\alpha D)$ as ideal sheaves.
\end{corollary}
It follows that the Hodge ideals $I_k(\alpha D)$ are always left continuous - a property not true in general (cf. \cite[Example 10.5]{MPbirational} for $f=x^2+y^3$).

The special multiplicity-free space arising from the coordinate-wise scaling action of $G=(\C^*)^n$ on $X=\C^n$ corresponds to the normal crossing case. Here \cite[Proposition 3.5]{Saito90} can, in principle, be recovered from Corollary \ref{corollary: G decomposition of grV} and our general results (Proposition \ref{prop:hodgeVgen} and Theorem \ref{thm:FS}). 

\subsection{Spaces of matrices and other examples}
Finally, we apply our results to compute the relevant structures for classical semi-invariant functions: the determinant of a generic (symmeric) matrix, the Pfaffian and the Freudenthal cubic on the fundamental representation of $E_6$.

For the determinant, our results extend the work of Perlman--Raicu \cite{PerlmanRaicu}. Let $X$ be the space of $n\times n$ matrices and let $D$ be the divisor of singular matrices. In \cite[Theorem 1.1]{PerlmanRaicu}, the Hodge ideal $I_k(D)$ is expressed in terms of symbolic powers of determinantal ideals, which we recover using our methods. Moreover, Corollary \ref{corollary: G decomposition of Hodge ideals} implies:
\begin{corollary}\label{for: Hodge ideals and higher multiplier ideals for determinants}
Let $k\in \N$ and $\alpha\in (0,1]$, then  $\tilde{I}_k(\alpha D)= I_k(\alpha D)=I_k(D)$.\end{corollary} 

Similarly, in Corollary \ref{cor: hodge ideal symmetric}, Corollary \ref{cor: Hodge ideal Pfaffian} and Corollary \ref{cor: Hodge ideal E6}, we compute these ideals for other semi-invariant functions using symbolic powers. It is noteworthy that, in the case of the determinant of a generic symmetric matrix, one has $\tilde{I}_1\left((\frac{1}{2}+\epsilon)D\right)\neq \cO_X$ for $0<\epsilon\ll 1$ (Remark \ref{remark: first hodge ideal for symmetric}), in which case the methods of \cite{PerlmanRaicu} likely do not apply.

On the other hand, Corollary \ref{thm: Gdecomposition of weight filtration} provides a new method for computing the characters of the simple components of $(\cO_X)_f\cdot f^{-\alpha}$ in a uniform matter, without using resolutions of singularities as in \cite{characters}. In fact, we can also compute the characters of the Hodge filtrations on these simples, generalizing the generic determinantal case of \cite{PerlmanRaicu}. We calculate this explicitly in the case of the Pfaffians (Corollary \ref{cor: simple for Phaffian}) and leave the details in the symmetric case to the interested readers. Furthermore, Corollary \ref{corollary: G decomposition of grV} recovers the nearby and vanishing cycle perverse sheaves computed in \cite[Theorem 4.12]{BG}, and lifts them to the level of mixed Hodge modules (see Remark \ref{remark: nearby cycle perverse sheaf for generic determinant}).

\smallskip

While the examples discussed above are multiplicity-free spaces, the results apply to many other interesting settings - for example, to semi-invariants and algebraic functions on prehomogeneous vector spaces. In such settings, our results can be also applied to the study of categories of equivariant $\sD$-modules (e.g. \cite{LW19}), and related objects and invariants, such as local cohomology modules \cite{LR20}, (co)homology, and various weight filtrations \cite{LR23, Perlman}; see Section 
\ref{sec:last} for further details.

\subsection*{Acknowledgement} We thank Bradley Dirks, Dougal Davis and Michael Perlman for helpful discussions.

\section{Recollection on the $V$-filtration of Kashiwara and Malgrange}
Let $X$ be a smooth algebraic variety, let $f$ be a regular function on $X$. Denote by $\iota: X \to X\times \C$ the graph embedding sending $x$ to $(x,f(x))$. Let $D=\mathrm{div}(f)$, and let $j:U=X-D\to X$ be the open embedding. Let $\cM$ be a regular holonomic $\sD_X$-module with quasi-unipotent monodromy along $f$. Let $t$ be the coordinate on $\C$ and set $s\colonequals-\d_t t$. 

\subsection{$V$-filtrations and $b$-functions}\label{sec: definition of V filtration}
\begin{thm}[{\cite{Kas83,Malgrange}, see also \cite[(4.1)]{Sabbah}}]\label{thm: uniqueness of V filtration}
    There exists a unique exhaustive decreasing filtration on $\iota_{+}\cM$, called the $V$-filtration along $f$, satisfying:
    \begin{enumerate}
        \item $V^{\bullet}\iota_{+}\cM$ is indexed left-continuously and discretely by rational numbers;
        \item For each $\alpha\in \Q$, $V^{\alpha}\iota_{+}\cM$ is a coherent $\sD_X\langle s,t\rangle$-module and satisfies: $t\cdot V^{\alpha}\iota_{+}\cM \subseteq V^{\alpha+1}\iota_{+}\cM$, with equality for $\alpha\gg 0$, and $\d_t\cdot V^{\alpha}\iota_{+}\cM\subseteq V^{\alpha-1}\iota_{+}\cM$;
    \item The operator $N:=s+\alpha$ acts nilpotently on $\gr_V^{\alpha}\iota_{+}\cM \colonequals V^\alpha \iota_+\cM / V^{>\alpha} \iota_+ \cM$.
    \end{enumerate}
\end{thm}
Assume that $f$ acts bijectively on $\cM$; equivalently, $\cM=\cM_f$, the localization of $\cM$ along $f$. Following \cite{MPVfiltration}, one can endow $\cM[s]f^s$ with a $\sD_X\langle s,t\rangle$-module structure via:
\begin{align*}
    \xi\cdot(ms^{\ell}f^s)&=(\xi(m)s^{\ell}+ms^{\ell+1}\frac{\xi(f)}{f})f^s, \quad \forall \xi\in T_X,\\
    s\cdot (ms^{\ell}f^s)&=ms^{\ell+1}f^s,\quad t\cdot (ms^{\ell}f^s)=\left(m(s+1)^{\ell} f\right) f^s.
\end{align*}
Define $\cM\cdot f^{-\alpha}\colonequals  \cM[s]f^s/(s+\alpha)$, and identify $\iota_{+}\cM \cong \sum_{\ell\geq 0} \cM\otimes \d_t^{\ell}$.

\begin{prop}[\cite{Malgrange,MPVfiltration}]\label{prop:malgrange}There is an isomorphism of $\sD_X \langle s,t \rangle$-modules
\begin{equation}\label{eqn: Malgrange isomorphism}
\cM[s]f^s \xrightarrow{\sim} \iota_+ \cM, \quad ms^{\ell}f^s\mapsto m\otimes (-\d_tt)^{\ell},
\end{equation}
with inverse given by
\begin{equation}\label{eqn: inverse of Malgrange isomorphism} m\otimes \d_t^{\ell} \to \frac{m}{f^{\ell}}\prod_{j=0}^{\ell-1}(-s+j)f^s.\end{equation}
\end{prop}
\begin{definition}\label{definition: definition of BS polynomials}
For a local section $w\in \iota_{+}\cM$, the Bernstein-Sato polynomial $b_w(s)$ is the unique monic polynomial of minimal degree $b(s)\in\C[s]$ satisfying
\begin{align}\label{eqn: definition of b-function}
     P(s,t)\cdot (tw)=b(s)\cdot w,
\end{align}
for some $P(s,t)\in \sD_X\langle s,t\rangle$. For $m\in \cM$ and $p(s)\in \C[s]$, define
\[ b_{p(s)m}(s):= b_w(s), \quad \textrm{where $w=p(-\d_tt)\cdot (m\otimes 1)$}.\]
\end{definition}
\begin{remark}
    For $1\in \cM$, $b_1(s)$ denotes the usual $b$-function of $f$, and \emph{not} $b_f(s)$.
\end{remark}
\begin{lemma}[{\cite[Lemma 1.21]{holoprehom}}]\label{lem:bdiv} For $m \in \cO_X$, we have $b_{m}(s) \mid \prod_{i=0}^{\deg m} b_1(s+i)$.
\end{lemma}

\begin{thm}[{\cite{Sabbah}, \cite[Corollary 2.7]{DLY}}]\label{thm: Sabbah}
For any $\alpha\in \mathbb{Q}$, we have
    \[ V^{\alpha}\iota_{+}\cM=\{w \in \iota_{+}\cM \mid \textrm{ all the roots of $b_w(s)$ are $\leq -\alpha$}\}.\]
Moreover, for $w\in V^{\alpha}\iota_{+}\cM$, one has
\begin{equation}\label{eqn: mult via ker}\mathrm{mult}_{s=-\alpha}b_{w}(s) = \min\{ \ell \mid (s+\alpha)^{\ell}\cdot [w]=0\in \mathrm{gr}^{\alpha}_V\iota_{+}\cM\}.\end{equation}
\end{thm}

\subsection{Weight filtrations}\label{sec: weight filtrations}
Let $W(N)_{\bullet}\gr^{\alpha}_V\iota_{+}\cM$  denote the monodromy weight filtration centered at $0$, which satisfies the convolution formula (see e.g. \cite[(23)]{DM05})
    \begin{equation}\label{eqn: convolution formula}
    W(N)_\ell = \sum_{i+j=\ell} \ker N^{i+1}\cap \mathrm{Im} N^{-j}.
    \end{equation}
Assume $\cM=\cS_f$, where $\cS$ is a simple $\sD$-module. By \cite[Theorem 19.4.1]{Mochizuki}, $\cS$ underlies a pure twistor $\sD$-module, say of weight $q$. One can show that the weight filtration on $\cM\cdot f^{-\alpha}$ satisfies $\gr^W_{\ell}(\cM\cdot f^{-\alpha})=0$ if $\ell\not\in [q,q+\dim X]$, and 
\begin{equation}\label{eqn: definition of Salpha}W_{q}(\cM\cdot f^{-\alpha})= j_{!\ast}(j^{\ast}\cS\cdot f^{-\alpha})=:\cS^{-\alpha}.\end{equation}
Consider the evaluation map
\begin{equation}\label{eqn: evaluation map} \mathrm{ev}_{s=-\alpha}: V^{\alpha}\iota_{+}\cM \xrightarrow{\eqref{eqn: inverse of Malgrange isomorphism}} \cM[s]f^s \xrightarrow{s=-\alpha}\cM\cdot f^{-\alpha}.\end{equation}
Since $\cS$ is simple, it follows from \cite[Proposition 5.11]{DLY} that \begin{equation}\label{eqn: ev sends V>alpha to Salpha}\mathrm{ev}_{s=-\alpha}(V^{>\alpha}\iota_{+}\cM)=\cS^{-\alpha}.\end{equation}
Therefore, $\mathrm{ev}_{s=-\alpha}$ induces an exact sequence \footnote{If $\cS$ underlies a pure Hodge module, this coincides with \cite[(2.24.3)]{Saito90}) via $\gr^{\alpha}_V\iota_{+}\cM\cong \gr^0_V\iota_{+}(\cM\cdot f^{-\alpha})$. See also \cite{MM}.}
\begin{equation}\label{eqn: evaluation map on grV}
\gr^{\alpha}_V\iota_{+}\cM \xrightarrow{N} \gr^{\alpha}_V\iota_{+}\cM\xrightarrow{ev} \cM\cdot f^{-\alpha}/\cS^{-\alpha}\to 0.\end{equation}
For $\ell>0$, the strictness of the weight filtration then yields a surjective map
\begin{equation}\label{eqn:evaluation map on W(N)} W(N)_{\ell-1} \gr^{\alpha}_V\iota_{+}\cM \twoheadrightarrow W_{q+\ell}(\cM\cdot f^{-\alpha})/\cS^{-\alpha}, \end{equation}
and an isomorphism (cf. \cite[(2.11.10)]{Saito90})
\begin{equation}\label{eqn: associated graded of weight} \frac{\gr^{W(N)}_{\ell-1}\gr^{\alpha}_V\iota_{+}\cM}{N\cdot \gr^{W(N)}_{\ell+1}\gr^{\alpha}_V\iota_{+}\cM}\xrightarrow{\sim} \gr^W_{q+\ell}(\cM\cdot f^{-\alpha}).\end{equation}

\subsection{Some results on $b$-functions}
Assume that $f$ acts on $\cM$ bijectively. We investigate the relation between $b_{p(s)w}(s)$ and $b_w(s)$. Recall the polynomial $[s+r]_k$ from \eqref{eqn: [s+a]}. 

\begin{lemma}\label{lem:conversediv}
For $w \in \iota_{+}\cM$ and $r \in \C$, there exists $k\in\Z_{>0}$ such that 
\[b_{(s+r)w}(s) \, \mid \,\, [s+r+1]_k \cdot b_w(s).
\]
\end{lemma}
\begin{proof}
It follows from the proof of \cite[Lemma 2.6]{DLY}.
\end{proof}
\begin{lemma}\label{lemma: basic property of q(s)w} 
For $w \in V^{\alpha}\iota_{+}\cM$ with $\alpha \in \Q$ and $q(s)\in \C[s]$ with $q(-\alpha)\neq 0$, we have
    \[\mathrm{mult}_{s=-\alpha}b_{q(s)w}(s) = \mathrm{mult}_{s=-\alpha}b_{w}(s).\]
   \end{lemma}

\begin{proof}
By \cite[Lemma 2.5]{DLY}, we have $\textrm{mult}_{s=-\alpha}b_{w}(s)\leq \textrm{mult}_{s=-\alpha}b_{q(s)w}(s).$ The reverse inequality follows from Theorem \ref{thm: Sabbah} and Lemma \ref{lem:conversediv}.
   \end{proof}

For $p(s), q(s) \in \C[s]$, we introduce the notation
\begin{equation}\label{eqn: notation cpq} c_{p,q}(s) = \gcd_{i\in \Z_{\geq 0}}\left( p(s+i) \cdot \prod_{j=0}^{i-1} q(s+j) \right).\end{equation}

\begin{prop}\label{prop:bgcd}
Let $m\in \cM$ and $p(s) \in \C[s]$. Then 
\begin{equation}\label{bpsm divides bm with another factor}b_{p(s) m}(s) \, \mid \, b_m(s) \cdot \dfrac{c_{p,b_m}(s+1)}{c_{p,b_m}(s)}.\end{equation}
In particular, $\deg b_{p(s)m}(s) \leq \deg b_m(s)$.
\end{prop}

\begin{proof}
Consider an equation yielding the Bernstein-Sato polynomial of $m$:
\[P(s) \cdot mf^{s+1} = b_{m}(s) \cdot mf^s,  \quad \textrm{with $P(s)\in \sD_X[s]$}.\]
For each $i\in \Z_{\geq 0}$, we obtain
\begin{equation}\label{eq:beq}
(P(s)\cdot t)^{i+1} \cdot (p(s) m f^s) = p(s+i+1) \prod_{j=0}^i b_{m}(s+j) \cdot m f^s.
\end{equation}
Let $d(s)$ denote the greatest common divisor of the elements $p(s+i+1)\prod_{j=0}^ib_m(s+j)$ for all $i\in \Z_{\geq 0}$. By Bezout's theorem, there exists a $\C[s]$-linear combination of these equations yielding:
\begin{equation}\label{eq:beq2}
Q(s,t) \cdot t \cdot (p(s)mf^s) = d(s) \cdot m f^s, \quad \textrm{for some $Q(s,t) \in \sD_X\langle s, t \rangle$}.
\end{equation}
Let $d'(s)=d(s)/\gcd(p(s),d(s))$ and $p'(s)=p(s)/\gcd(p(s),d(s))$. Multiplying (\ref{eq:beq2}) by $p'(s)$ gives:
\[p'(s)Q(s,t) \cdot t \cdot (p(s)mf^s) = d'(s) \cdot (p(s)m f^s), \]
which implies that $b_{p(s)m}(s) \mid d'(s)$. A direct computation shows that $d(s)=b_m(s)\cdot c_{p,b_{m}}(s+1)$ and $\gcd(p(s),d(s)) = c_{p, b_m}(s)$, yielding the desired result.
\end{proof}

\begin{lemma}\label{lem:expbfun}
    Let $p(s), q(s)\in \C[s]$, and factor $q(s)=\prod_{i=1}^d(s+r_i)$. Choose the maximal $(k_1,\dots, k_d) \in \Z_{\geq 0}^d$ such that $\prod_{i=1}^d[s+r_i]_{k_i}$ divides $p(s)$. Then
    \[q(s) \cdot \dfrac{c_{p,q}(s+1)}{c_{p,q}(s)} = \prod_{i=1}^d(s+r_i + k_i).\]
\end{lemma}

\begin{proof}
We pursue induction on $d$; the case $d=0$ is trivial. Without loss of generality, assume that $\max\{r_i + k_i \mid 1\leq i \leq d\} = r_1+k_1$, and set $n=r_1+k_1$. Then,
\begin{equation}\label{eq:nodiv0}
(s+n+i)\, \nmid \, p(s+i) \prod_{j=0}^{i-1} q(s+j), \quad \textrm{for every $i\in \Z_{\geq 0}$}.
\end{equation}
Write $p(s) = [s+r_1]_{k_1} \cdot p'(s)$ and $q(s)=(s+r_1) \cdot q'(s)$. By (\ref{eq:nodiv0}) we have
\begin{align*}c_{p,q}(s) &= [s+r_1]_{k_1} \cdot \gcd_{i\in \Z_{\geq 0}}\left( [s+n]_i \cdot p'(s+i) \cdot \prod_{j=0}^{i-1} q'(s+j) \right) \\
&= [s+r_1]_{k_1} \cdot \gcd_{i\in \Z_{\geq 0}}\left( p'(s+i) \cdot \prod_{j=0}^{i-1} q'(s+j) \right)\\
&=[s+r_1]_{k_1} \cdot c_{p',q'}(s).\end{align*}
Therefore
\[q(s) \cdot \dfrac{c_{p,q}(s+1)}{c_{p,q}(s)} = (s+n) \cdot q'(s) \cdot\dfrac{c_{p',q'}(s+1)}{c_{p',q'}(s)},\]
and the result follows by induction.
\end{proof}

Now, for a non-zero $m \in \cM$ and $\alpha \in \Q$, define the set
\[I_{m, \alpha} :=   \{p(s) \in \C[s] \, \mid \,  p(s)mf^s\in V^{\alpha}\iota_{+}\cM\}\]
which is a non-zero ideal in $\C[s]$.
\begin{definition}\label{definition: pfunction of m}
For any $m\in \cM$, define the \emph{$p$-function} of $m$ with respect to $\alpha$, denoted by $p_{m,\alpha}(s)$, to be the monic polynomial $p(s)\in \C[s]$ of minimal degree such that $p(s) m f^s\in V^{\alpha}\iota_{+}\cM$; i.e. $(p_{m,\alpha}(s))=I_{m,\alpha}$. Define $\nu_{m,\alpha}=\mathrm{mult}_{s=-\alpha}b_{p_{m,\alpha}(s)m}(s)$.
\end{definition}

\begin{remark}\label{remark: p function do not vanishing on -alpha}
One can construct $p_{m,\alpha}(s)$ inductively using the fact that $(s+\beta)$ acts nilpotently on $\gr^{\beta}_V\iota_{+}\cM$, see \cite{LY25}. It follows that $p_{m,\alpha}(-\alpha)\neq 0$.
\end{remark}
\begin{lemma}\label{lemma: basic property of pfunction}
   Let $r(s)\in \C[s]$ satisfy $r(-\alpha)\neq 0$. Then
    \[ \mathrm{mult}_{s=-\alpha}b_{r(s)p_{m,\alpha}(s)m}(s)=\mathrm{mult}_{s=-\alpha}b_{p_{m,\alpha}(s)m}(s).\]
\end{lemma}

\begin{proof}
    This follows from Lemma \ref{lemma: basic property of q(s)w}, since $p_{m,\alpha}(s)mf^s\in V^{\alpha}\iota_{+}\cM$.
\end{proof}

\section{Equivariant setup}
\label{sec:equiv}

In this section, let $X$ be an affine smooth algebraic variety equipped with an action by a connected linear algebraic group $G$.  Let $\cM$ be a (strongly) \emph{$G$-equivariant} $\sD_X$-module; that is, $\cM$ is holonomic and there exists a $\sD_{G\times X}$-isomorphism $\tau:  m^{\ast}\cM\xrightarrow{\sim} p^{\ast}\cM$, where $p$ and $m$ are the projection and multiplication maps from $G\times X$ to $X$, and $\tau$ satisfies the usual compatibility conditions on $G\times G\times X$. For the more general concept of twisted equivariant $\sD$-modules, see \cite{hotta}.
\begin{definition}\label{definition: semiinvariant of weight sigma}
A regular function $f$ on $X$ is called \emph{semi-invariant with respect to $G$ of weight $\sigma\in \mathrm{Hom}(G,\C^{\ast})$} if $g\cdot f=\sigma(g)f$. In other words,
\[ f(g\cdot x)=\sigma(g)^{-1} f(x), \quad \forall x\in X,g\in G.\]
\end{definition}
From now on, fix a semi-invariant function $f$ of weight $\sigma$. By a slight abuse of notation, we also denote the corresponding character of the Lie algebra $\lie$ by $\sigma$.

\subsection{$G$-equivariance of $V$-filtrations and differential operators}\label{sec: Gequivariant diff}
The function $f$ induces a unique $G$-action on $X\times \C$ via 
\[g\cdot (x,t)\colonequals (g\cdot x, \sigma(g)^{-1}t), \quad \forall g\in G,\] 
making $\iota:X\to X\times \C$ a $G$-equivariant map and turning $t:X\times \C\to \C$ into a semi-invariant function of weight $\sigma$. In particular, $\iota_{+}\cM$ is a $G$-equivariant $\sD_{X\times \C}$-module. Denote by $V^{\bullet}\iota_{+}\cM$ the $V$-filtration along $f$.

\begin{lemma}\label{lemma: G and s compatible on V filtration}
For each $\alpha$, $V^{\alpha}\iota_{+}\cM$ is a $G$-equivariant $\cO_{X \times \C}$-submodule of $\iota_{+}\cM$, i.e.
\[ g(V^{\alpha}\iota_{+}\cM)\subseteq V^{\alpha}\iota_{+}\cM, \quad \forall g\in G.\]
\end{lemma}

\begin{proof}
Let us consider the filtration $g(V^{\bullet}\iota_{+}\cM)$
on $\iota_{+}\cM$. It is easy to check that it satisfies all the conditions in Theorem \ref{thm: uniqueness of V filtration} with respect to the function $f$: the property (1) follows from the fact that $g$ is an automorphism of $\iota_{+}\cM$. For property (2), one uses the fact that $g\cdot t=\sigma(g)t$ and so
\[ g(t\cdot V^{\alpha}\iota_{+}\cM)=(g\cdot t)\cdot g(V^{\alpha}\iota_{+}\cM)=\sigma(g)t\cdot g(V^{\alpha}\iota_{+}\cM) = t \cdot g(V^{\alpha}\iota_{+}\cM). \]
The part with $\partial_t$ is analogous as $g \cdot \partial_t = \sigma(g)^{-1} \partial_t$. The property (3) follows in a similar way. Consequently, by the uniqueness of $V$-filtration we have
\[ g(V^{\alpha}\iota_{+}\cM)=V^{\alpha}\iota_{+}\cM,\quad \textrm{for all $\alpha$}. \qedhere\]\end{proof}
\begin{remark} 
The isomorphism $\iota_{+}\cM\cong \cM[s]f^s$ from \eqref{eqn: inverse of Malgrange isomorphism} is a morphism of (strongly) $[G,G]$-equivariant $\sD_X$-modules, though not necessarily of $G$-equivariant $\sD_X$-modules. Nevertheless, by viewing $f^s$ as an invariant element, the isomorphism becomes an isomorphism of weakly $G$-equivariant modules.
\end{remark}

\begin{prop}\label{prop: quotient of weight are equivariant}
For any $\alpha\in \Q,\ell\in \Z$, the quotient $W(N)_{\ell}\gr^{\alpha}_V\iota_{+}\cM/W(N)_{\ell-2}\gr^{\alpha}_V\iota_{+}\cM$ is a $-\alpha\sigma$-twisted equivariant $\sD$-module.
\end{prop}

\begin{proof}
    Let $a: \lie \to \sD_X$ be the Lie algebra map induced by the action of $G$ on $X$ on the level of operators. Then differentiating the $G$-module action on $\cM[s]f^s$ agrees with the action of the operators $a(\xi) - s \cdot \sigma(\xi) \in \sD_X[s]$, with $\xi \in \mathfrak{g}$. Since $N \cdot W_{\ell}\gr^{\alpha}_V\iota_{+}\cM \subseteq  W_{\ell-2}\gr^{\alpha}_V\iota_{+}\cM$, the latter operators descend to $a(\xi) + \alpha \cdot \sigma(\xi) \in \sD_X$ on the quotient. Thus, differentiating the $G$-module action on the quotient $W(N)_{\ell}\gr^{\alpha}_V\iota_{+}\cM/W(N)_{\ell-2}\gr^{\alpha}_V\iota_{+}\cM$ agrees with the action of the vector fields in $\mathfrak{g}$, up to the twist of the character $-\alpha \sigma$, which is equivalent to twisted equivariance (see \cite{hotta}; for the case $\alpha = 0$, cf. \cite[Proposition 2.6]{vanderbergh}).
\end{proof}
\begin{prop}\label{prop:invdiff}
    Assume $G$ is reductive, and let $w\in \iota_{+}\cM$ be such that its $G$-span is irreducible. Then the differential operator $P(s,t)$ computing $b_w(s)$ as in \eqref{eqn: definition of b-function} can be chosen to be a $G$-semi-invariant of weight $-\sigma$.
\end{prop}

\begin{proof}
Recall the equation
\begin{equation}\label{eq:turninv}
    P(s,t)\cdot (tw) = b_w(s)\cdot w, \quad \textrm{for some $P(s,t) \in \sD_X\langle s,t\rangle$}.
\end{equation}
Let $G'=[G,G]$ be the derived subgroup of $G$, which is semisimple and connected. By construction $g\cdot t= \sigma(g)\cdot t$, so $t$ is $G'$-invariant. Since $\iota_{+}\cM$ is $G$-equivariant, we can choose an irreducible $G$-submodule $W \subseteq \iota_{+}\cM$ containing $w$, which is irreducible also as a $G'$-module.

We first show that $P(s,t)$ in \eqref{eq:turninv} can be chosen to be $G'$-invariant. Let $\{w_i\}$ be a basis of $W$, and let $E$ be the image of the universal enveloping algebra of $G'$ inside $\sD_X$. By the Jacobson Density Theorem \cite[Section 4.3]{jacobson}, for each $i$ there exists $\xi_i \in E$ such that $\xi_i \cdot w_i = w$ and $\xi_i \cdot w_j = 0$ for $j\neq i$. Furthermore, there exist elements $\eta_i \in E$ such that $\eta_i \cdot w = w_i$.

Define the operator $Q = \sum_{i} \eta_i \cdot P \cdot \xi_i 
 \,\, \in\sD_X\langle s,t\rangle$. For each $i$, we have $Q \cdot (tw_i) = b_w(s) \cdot w_i$. Applying the Reynolds operator $\sD_X\langle s,t\rangle \to \sD_X^{G'}\langle s,t\rangle $ (i.e., averaging over $G$) yields an operator $Q'\in \sD_X^{G'}\langle s,t\rangle$ such that $Q' \cdot (tw_i) = b_w(s) \cdot w_i$, for all $i$. In particular, $Q' \cdot (tw) = b_w(s) \cdot w$.

Now consider the $G$-span of $Q'$ in $\sD_X^{G'}\langle s,t\rangle$, denoted $Z$, which is a $G/G'$-module. Since $G/G'$ is a torus, we can decompose $Z= \oplus Z_\chi$ into a direct sum $G$-modules with weights corresponding to characters $\chi$ of $G$.  Write $Q' = \sum_i Q_\chi = Q_{-\sigma} + Q''$, where $Q_{-\sigma}$ is the summand of weight $-\sigma$. Since $W$ is an irreducible $G$-module and $t$ is semi-invariant of weight $\sigma$, the only summand of $Q'$ contributing in the $G$-isotypic component of $Q'\cdot (tw)$ corresponding to $W$ is $Q_{-\sigma}$. Since $Q' \cdot (tw) = b_w(s) \cdot w$, we get that $Q''\cdot (tw)=0$ and $Q_{-\sigma} \cdot (tw) = b_w(s) \cdot w$. Therefore $Q_{-\sigma}$ is the desired $G$-semi-invariant operator.
\end{proof}

\subsection{Multiplicity-free spaces}
Assume that $(G,X)$ is a multiplicity-free space. 
\begin{lemma}\label{lem:multfree}
 Let $\cS$ be a simple equivariant $\sD$-module with $\mathrm{supp}(\cS)\not\subset f^{-1}(0)$, and let $\cM=\cS_f$. Then $\cM$ admits a multiplicity-free decomposition as a $G$-module.
\end{lemma}

\begin{proof}
Let $j:U=f^{-1}(\C^{\times})\to X$ be the open embedding. Then $j^{\ast}\cM$ is a simple equivariant $\sD_U$-module.  Since $U$ is spherical, it follows from \cite[Theorem 3.17 (a)]{LW19} that the space of sections of $j^{\ast}\cM$ decomposes multiplicity-freely as a $G$-module. Hence so does $\cM=j_*j^*\cM$. \end{proof}
\begin{lemma}\label{lem:decomp}
Let $P \in \sD_X$ be a semi-invariant operator of weight $-\sigma$. Then there exists a $G$-invariant operator $Q \in \sD_X$ such that $P=Q \cdot \partial f$. 
\end{lemma}

\begin{proof}
It is known (cf. Lemma \ref{lem:multfree}) that there is a multiplicity-free decomposition 
\[\C[X] = \Sym X^* = \bigoplus_{\lambda\in \Lambda(\cO_X)} U_\lambda.\]
This implies a $G$-decomposition of the ring of differential operators:
\[\sD_X = \Sym X^* \otimes \Sym X = \bigoplus_{\alpha, \beta} U_\lambda \otimes U_\beta^*. \]
We aim to describe the semi-invariants of weight $-\sigma$ among these summands; equivalently, $G$-invariant elements in $U_{\lambda+\sigma} \otimes U_{\beta}^*$. By Schur's lemma, 
\[(U_{\lambda+\sigma} \otimes U_{\beta}^*)^G= \op{Hom}_G (U_\beta, U_{\lambda+\sigma})\neq 0 \Longleftrightarrow \beta = \lambda+\sigma.\]
Therefore, any semi-invariant operator $P\in \sD$ of weight $-\sigma$ can be written as: 
\[P= \sum_\lambda P_\lambda, \mbox{ with } P_\lambda \in U_{\lambda} \otimes U_{\lambda+\sigma}^*.\]
Now, since $\Sym X^*$ is multiplicity-free, multiplication by $f$ induces an isomorphism \[f\cdot U_{\lambda} \cong U_{\lambda+\sigma} \textrm{ in $\C[X]=\Sym X^*$}.\]
Dually, we have $U^*_{\lambda +\sigma}= U^*_\lambda \cdot \partial f$ in $\Sym X$, from which the conclusion follows.
\end{proof}

\subsection{The case of irreducible isotypic components}\label{sec:reductive}
Assume $G$ is \emph{reductive}, and let $\Lambda$ be the set of integral dominant weights, parameterizing the set of finite-dimensional irreducible $G$-modules. Let $\cS$ be a simple equivariant $\sD$-module with $\textrm{supp}(\cS)\not\subset f^{-1}(0)$, and suppose $\cM=\cS_f$. For each $\lambda\in \Lambda$, let $\cM_{\lambda}$ denote the corresponding isotypic component in the $G$-decomposition.

\medskip

Fix $\lambda \in \Lambda$ such that $\cM_\lambda$ is \emph{irreducible}; which is equivalent to $\cM_{\lambda+k \sigma}$ being irreducible for all $k \in \Z$. We repeatedly use the following observation.
\begin{lemma}\label{lem:easysubmod}
    All $(G,\C[s])$-submodules of $\cM_\lambda[s]$ are of the form $(q(s)) \cdot \cM_\lambda$, for some $q(s)\in\C[s]$ and 
\[
    (\cM[s]f^s)_\lambda = \cM_\lambda[s]f^s.
\]
\end{lemma}

\begin{proof}
    Since we view both $s$ and $f^s$ as $G$-invariants, the latter claim is clear. 
    
    For the former, let $U \subset \cM_\lambda[s]$ first be just a $G$-submodule. We claim that $U=\cM_\lambda \otimes W$, for some subspace $W\subset \C[s]$. As $\cM_\lambda[s]$ is a semisimple $G$-module, so is $U$, hence a sum of its simple submodules. Therefore, it is enough to prove the claim for $U$ simple. In this case, $U$ is finite dimensional, hence there is a finite-dimensional subspace $W' \subset \C[s]$ such that $U\subset \cM_\lambda \otimes W'$. Clearly, the natural map $\mathrm{Hom}_G(U, \cM_\lambda)\otimes_\C W' \to \mathrm{Hom}_G(U, \cM_\lambda\otimes_\C W')$ is an isomorphism. By Schur's lemma, $\mathrm{Hom}_G(U, \cM_\lambda) = \C$, hence $U=\cM_\lambda \otimes W$, with $W \subset W'$ and $\dim W=1$. 

    Now let $N\subset \C[s]$ be a $(G,\C[s])$-module. By the previous claim, we can write $N=\cM_\lambda \otimes W$, for some subspace $W\subset \C[s]$. Since $N$ is also stable under $\C[s]$, so must be $W$, thus the conclusion.    
\end{proof}

\begin{lemma}\label{lemma: consequence of full support} If $\supp \cS=X$, then $(\cO_X)_{k\sigma} = \C \cdot f^k$, for all $k \in \Z_{\geq 0}$.
\end{lemma}

\begin{proof}
    Suppose there exists  $k \in  \Z_{\geq 0}$ and $g\in (\cO_X)_{k\sigma}$ with $g \notin \C \cdot f^k$. Let $m_{\lambda}$ be a highest weight vector in $\cM_\lambda$. Then both $g \cdot m_{\lambda}$ and $f^k \cdot m_{\lambda}$ are highest weight vectors in $\cM_{\lambda+k\sigma}$, which are irreducible by assumption. Hence there exists $c \in \C$ such that $g m_{\lambda} = c f^k m_{\lambda}$. Since $\supp \cS = X$, this forces $g=c f^k$, contradicting the assumption.
\end{proof}

\begin{prop}\label{prop:bfunirred} Assume either $\supp \cS = X$ or that $X$ is a multiplicity-free space. Then there exist a unique semi-invariant differential operator $\partial f$ with constant coefficients of weight $-\sigma$ and $\partial f \cdot f=1$, and a monic polynomial $b_{\lambda}(s)$ with $\deg b_{\lambda}(s) = \deg f$, such that for any $0\neq m \in \cM_{\lambda}$ one has
\[ \partial f \cdot m f^{s+1} = b_{\lambda}(s) \cdot m f^s, \]
and $b_{\lambda}(s)=b_m(s)$.
\end{prop}

\begin{proof}
By Lemma \ref{lemma: consequence of full support}, we have $(\cO_X)_\sigma = \C \cdot f$. For the explicit construction of $\partial f$, see \cite[Section 4, Proposition 24]{saki}. Consider the $G$-invariant operator $Q=\partial f \cdot f \in \sD_X$, and restrict to the $\lambda$-isotypic component $\cM_{\lambda}$. This induces a $(G,\C[s])$-module map
\[Q_\lambda :  (\cM[s]f^s)_\lambda \to (\cM[s]f^s)_\lambda.\]
By Lemma \ref{lem:easysubmod}, the image of $Q$ is of the form $(b_\lambda(s))\cdot \cM_\lambda f^s$, with $b_{\lambda}(s)$ monic. Using Schur's lemma and the irreduciblity of $\cM_{\lambda}$, this implies 
\[ \partial f \cdot m f^{s+1} =c\cdot b_\lambda(s)\cdot m f^s\]
for some $c\in \C$. As in the proof of \cite[Theorem 2.18]{holoprehom} (see also \cite[Section 4]{saki}), we can deduce that $\deg b_{\lambda}(s) = \deg f$ and $c=1$, hence $b_m(s) \mid b_\lambda(s)$.

If $\supp(\cS)=X$, the reverse divisibility $b_\lambda(s)\mid b_{m}(s)$ follows as in \cite[Theorem 2.18]{holoprehom}. If instead $X$ is multiplicity-free, based on Proposition \ref{prop:invdiff}, there exists a $G$-semi-invariant operator $P(s) \in \sD_X[s]$ of weight $-\sigma$ such that  $P(s) \cdot  m f^{s+1} = b_{m}(s) \cdot m f^s$. Using Lemma \ref{lem:decomp}, we can write $P(s)= R(s) \cdot \partial f$ for some $G$-invariant operator $R(s) \in \sD_X[s]$. It follows that $b_\lambda(s) \, | \, b_{m}(s)$.
\end{proof}

\begin{thm}\label{thm:bfun}
For $0\neq m\in \cM_\lambda$ and $p(s)\in \C[s]$, then
\[ b_{p(s)m}(s)=b_{m}(s)\cdot\frac{c_{p,b_{m}}(s+1)}{c_{p,b_{m}}(s)},\]
where $c_{p,q}(s)$ is as in \eqref{eqn: notation cpq}. In particular, $\deg b_{p(s)m}(s) = \deg b_m(s)$. Moreover, if $b_m(s)=\prod_{i=1}^d(s+r_i)$ and $\prod_{i=1}^d [s+r_i]_{k_i}$ divides $p(s)$ is maximal for some $(k_1,\dots, k_d) \in \Z_{\geq 0}^d$, then $b_{p(s)m}(s) = \prod_{i=1}^d (s+r_i+k_i)$.
\end{thm}

\begin{proof}
Proposition \ref{prop:bgcd} shows that $b_{p(s)m}(s)$ divides $b_{m}(s)\cdot\frac{c_{p,b_{m}}(s+1)}{c_{p,b_{m}}(s)}$. To show the reverse divisibility, note that $\cM_{\lambda}[s]f^s$, being the $G$-span of $p(s)mf^s$ in $\cM[s]f^s=\iota_{+}\cM$, is irreducible. Hence by Proposition \ref{prop:invdiff}, there exists a semi-invariant operator $P \in \sD_X\langle s, t \rangle$ of weight $-\sigma$ such that 
\begin{equation}\label{eq:bsection3}
    P \cdot p(s+1)m f^{s+1} = b_{p(s)m}(s) \cdot p(s) m f^s.
\end{equation}

Decomposing $P=\sum_{i\geq 0} P_i \cdot t^i$ with each $P_i \in \sD_X[s]$ semi-invariant of weight $-(i+1)\sigma$, as in the proof of Proposition \ref{prop:bfunirred}, we can use Schur's lemma to obtain
\[P_i(s) \cdot mf^{s+i+1} = q_i(s) \cdot mf^{s},\]
for some $q_i(s) \in \C[s]$. Substituting $s\mapsto (i+1)\cdot s$, we get
\[P_i((i+1)s) \cdot m(f^{i+1})^{s+1} = q_i((i+1)s) \cdot m(f^{i+1})^{s}.\]
Applying Proposition \ref{prop:bfunirred} to the function $f^{i+1}$ instead of $f$, we obtain that $q_i(s)$ is divisible by $\prod_{j=0}^i b_m(s+j)$. Denote by $q'_i(s)=q_i(s)/ \prod_{j=0}^i b_m(s+j)$. 

Then equation (\ref{eq:bsection3}) becomes
\begin{equation*}
\sum_{i\geq 0} q'_i(s) \cdot\left(p(s+i+1)\prod_{j=0}^{i}b_{m}(s+j)\right)\cdot m f^s = b_{p(s)m}(s)\cdot p(s) m f^s.
\end{equation*}
Thus, $b_{p(s)m}(s)p(s)$ lies in the ideal of $\C[s]$ generated by the elements $p(s+i+1)b_m(s)b_m(s+1)\dots b_m(s+i)$, with $i\in \Z_{\geq 0}$, whose greatest common divisor is $b_{m}(s)\cdot c_{p,b_{m}}(s+1)$.  It follows that
\[ \frac{b_{m}(s)\cdot c_{p,b_{m}}(s+1)}{c_{p,b_{m}}(s)}=\frac{b_{m}(s)\cdot c_{p,b_{m}}(s+1)}{\gcd(p(s),b_{m}(s)\cdot c_{p,b_{m}}(s+1))}\mid b_{p(s)m}(s),\]
which proves the converse divisibility. The final formula follows from Lemma \ref{lem:expbfun}.
\end{proof}

\subsection{Proof of main results}
\begin{proof}[Proof of Theorem \ref{thm: V filtration general G}]
    By Sabbah's Theorem \ref{thm: Sabbah}, Theorem \ref{thm:bfun} implies that $p(s) \cdot m f^s \in (V^\alpha \iota_+ \cM)_{\lambda}$  if and only if $\prod_{i=1}^d [s+\lambda_i]_{\lceil \alpha - \lambda_i\rceil}$ divides $p(s)$. 
\end{proof}

\begin{remark}\label{rem:weighgrVlambda}
    As shown in the proof of Corollary \ref{corollary: G decomposition of grV}, the formula \ref{eqn: convolution formula} implies: if $\nu_{\lambda, \alpha} > -\ell$, then as $(\C[s], G)$-modules
    \[(W(N)_\ell \gr_V^\alpha \iota_+ \cM)_\lambda = \C[s]/((s+\alpha)^{\min\{\lfloor \frac{\nu_{\lambda,\alpha}+\ell+1}{2} \rfloor, \nu_{\lambda,\alpha}\}} \otimes \cM_\lambda\]
    and the $\lambda$-isotopic component vanishes if $\nu_{\lambda, \alpha} \leq -\ell$.
\end{remark}
\begin{proof}[Proof of Theorem \ref{thm:weightisot}]
Applying Theorem \ref{thm:bfun} to \eqref{eqn: plambdaalpha in general} shows that $\nu_{m,\alpha}$ from Definition \ref{definition: pfunction of m} satisfies \eqref{eqn: numalpha}. 

``$\Leftarrow$": Using the map \eqref{eqn: evaluation map on grV} and Remark \ref{remark: p function do not vanishing on -alpha}, we have
\[ ev\left(\left[\frac{p_{m,\alpha}(s)}{p_{m,\alpha}(-\alpha)}mf^s\right]\right)=[m\cdot f^{-\alpha}].\]
By \eqref{eqn: mult via ker} and \eqref{eqn: convolution formula}, we see that $\left[\frac{p_{m,\alpha}(s)}{p_{m,\alpha}(-\alpha)}mf^s\right]\in W(N)_{\nu_{m,\alpha}-1}\gr^{\alpha}_V\iota_{+}\cM$. Hence \eqref{eqn:evaluation map on W(N)} implies $m\cdot f^{-\alpha} \in W_{\nu_{m, \alpha}} \cM f^{-\alpha}$.

``$\Rightarrow$": From the convolution formula \eqref{eqn: convolution formula}, we get $W(N)_{\ell-1}\subseteq \Ker N^{\ell}+\mathrm{Im}(N)$. Using \eqref{eqn: evaluation map on grV}, the map \eqref{eqn:evaluation map on W(N)} descends to a surjective map:
\begin{equation}\label{eqn: evaluation map from Ker Nell} \Ker N^{\ell} \twoheadrightarrow W_{q+\ell}(\cM\cdot f^{-\alpha})/\cS^{-\alpha}.\end{equation}
After twisting by $\C_{\alpha\sigma}$, the map becomes $G$-equivariant. Taking the $\lambda$-isotypic component yields a surjection:
\[ ev_\lambda:(\Ker N^{\ell})_{\lambda} \twoheadrightarrow W_{q+\ell}(\cM\cdot f^{-\alpha})_{\lambda-\alpha \sigma}/(\cS^{-\alpha})_{\lambda-\alpha \sigma} \otimes \C_{\alpha\sigma}.\]

To prove $\ell\geq \nu_{m,\alpha}$ for $m\cdot f^{-\alpha}\in W_{q+\ell}(\cM\cdot f^{-\alpha})_{\lambda-\alpha \sigma}$, consider two cases. 

\textbf{Case 1}: If $m\cdot f^{-\alpha}\not\in \cS^{-\alpha}$, the class $[m\cdot f^{-\alpha}]$ is nonzero. Choose $[w] \in (\Ker N^\ell)_\lambda$ such that $ev_\lambda([w])=[m \cdot f^{-\alpha}]$. As $p_{\lambda,\alpha}(s)=p_{m,\alpha}(s)$, by Theorem \ref{thm: V filtration general G} we pick a representative 
\[ w = c(s) p_{m,\alpha}(s) m f^s \in (V^\alpha \iota_+ \cM)_\lambda\]
for some $c(s)\in \C[s]$ with $c(-\alpha)\neq 0$. Applying \eqref{eqn: mult via ker} and  Lemma \ref{lemma: basic property of pfunction} gives:
\[ \ell\geq \mathrm{mult}_{s=-\alpha} b_{c(s) p_{m,\alpha}(s) m }(s)=\nu_{m,\alpha}.\]

\textbf{Case 2}: If $m\cdot f^{-\alpha}\in \cS^{-\alpha}$, then the $G$-equivariance of \eqref{eqn: ev sends V>alpha to Salpha} gives a surjection:
\[ ev:(V^{>\alpha} \iota_+ \cM)_\lambda \twoheadrightarrow (\cS^{-\alpha})_{\lambda-\alpha \sigma} \otimes \C_{\alpha\sigma}.\]
As before, choose an element $c(s)p_{m,\alpha}(s) m f^s\in (V^{>\alpha} \iota_+ \cM)_\lambda$ mapping to $m\cdot f^{-\alpha}$. Since $c(-\alpha)\neq 0$, Lemma \ref{lemma: basic property of pfunction} and Theorem \ref{thm: Sabbah} imply $\nu_{m, \alpha} = 0\leq \ell$.

\end{proof}

\begin{corollary}\label{cor:sinm}
    For a nonzero $m\in \cM_{\lambda}$, $m \cdot f^{-\alpha}\,\in \cS^{-\alpha}$ from \eqref{eqn: definition of Salpha} if and only if $b_\lambda(s-\alpha)$ has no roots in $\Z_{\geq 0}$.
\end{corollary}

\begin{proof}
    Since $\cS^{-\alpha}=W_q(\cM\cdot f^{-\alpha})$ by \eqref{eqn: definition of Salpha}, this follows directly from Theorem \ref{thm:weightisot}.
\end{proof}
\begin{proof}[Proof of Corollary \ref{cor: ValphaiotaS}]
Note $V^\alpha \iota_+ \cS = V^\alpha \iota_+ \cM \bigcap \iota_+ \cS$. One can check that under the isomorphism (\ref{eqn: inverse of Malgrange isomorphism}) we have
\begin{equation*}
(\iota_+ \cS)_\lambda = ([s-(r-1)]_r) \cdot \cM_\lambda,
\end{equation*}
where $r$ is the smallest nonnegative integer such that $\lambda + r \sigma \in \Lambda(\cS)$. By Corollary \ref{cor:sinm}, this is equivalent to $b_{\lambda+r\sigma}(s)=b_{\lambda}(s+r)$ having no nonnegative roots. Thus $r=r_{\lambda}+1$ and the conclusion follows from Theorem \ref{thm: V filtration general G}.
\end{proof}

It follows that the nilpotency order of $s+\alpha$ on $\gr^{\alpha}_V\iota_{+}\cM$ is also determined by $b_\lambda(s)=\prod_{i=1}^{d} (s+\lambda_i)$ (compare with \cite[\S 4.7]{Gyoja1996}). In particular, any irreducible isotypic component of $\cM$ determines the nilpotency order.

\begin{prop}\label{prop:nilp}
    The nilpotency order of $N=s+\alpha$ on $\gr^{\alpha}_V\iota_{+}\cM$ is $ \# \{ i \mid \lambda_i \in \alpha+ \mathbb{Z}  \}$.
\end{prop}

\begin{proof}
Let $n_\alpha$ denote this nilpotency order. It is known that the weight length of $\cM \cdot f^{-\alpha}$ equals $n_\alpha$, i.e. $W_{q+n_\alpha-1} (\cM f^{-\alpha}) \neq \cM f^{-\alpha}$ and $W_{q+n_\alpha} (\cM f^{-\alpha}) = \cM f^{-\alpha}$. Take $m\in \cM_\lambda$ and let $m':= m f^{\alpha-k}$ for $k\gg 0$. Then $m'$ generates $\cM f^{-\alpha}$ as a $\sD_X$-module, and by Theorem \ref{thm:weightisot} we have $n_\alpha=\nu_{m', \alpha}$. Using (\ref{eqn: numalpha}) we obtain $\nu_{m',\alpha}= \# \{ i \mid \lambda_i \in \alpha+ \mathbb{Z}  \}$.
\end{proof}

Theorem \ref{thm:weightisot} also provides an explicit construction of simple composition factors using Bernstein-Sato polynomials (cf. \cite[Proposition 4.9]{LW19}).

\begin{prop}\label{prop:simples}
Let $\alpha\in \Q$. For a non-zero $m\in \cM_{\lambda}$, the $\sD$-module 
\[\sD_X \cdot m f^{-\alpha} / \sD_X \cdot m f^{-\alpha+1}\]
is non-zero if and only if $b_\lambda(-\alpha) = 0$. In that case, its unique simple quotient is 
$\sD_X  m f^{-\alpha} / (\sD_X m f^{-\alpha} \cap W_{q+\nu_{\lambda,\alpha}-1} (\cM  f^{-\alpha}))$, which is also a summand of $\gr^W_{q+\nu_{\lambda,\alpha}}( \cM f^{-\alpha})$.
\end{prop}

\begin{proof}
If $b_\lambda(-\alpha)\neq 0$, substituting $s=-\alpha$ into \eqref{eqn: definition of b-function} gives $\sD_X \cdot m f^{-\alpha+1}=\sD_X \cdot m f^{-\alpha}$. If $b_{\lambda}(-\alpha)=0$, Theorem \ref{thm:weightisot} tells us that
\[ m\cdot f^{-\alpha} \notin W_{q+\nu_{\lambda,\alpha}-1}(\cM\cdot f^{-\alpha}), \mbox{ and } m\cdot f^{-\alpha+1}\in W_{q+\nu_{\lambda+\sigma,\alpha}}(\cM\cdot f^{-\alpha}).\]
By \eqref{eqn: numalpha} one has $\nu_{\lambda,\alpha} - \nu_{\lambda+\sigma, \alpha} = \mathrm{mult}_{s=-\alpha} b_\lambda(s) \geq 1$, so $\sD_X \cdot m f^{-\alpha} / \sD_X \cdot m f^{-\alpha+1}\neq 0$.

The module $\sD_X \cdot m f^{-\alpha}$ is generated by $(\cM \cdot f^{-\alpha})_{\lambda-\alpha\sigma}$, which is irreducible as a $\lie$-module. The sum of proper $\sD_X$-submodules of $\sD_X \cdot m f^{-\alpha}$ must have a trivial $\lambda-\alpha\sigma$-isotypical component, making it the unique maximal submodule. Thus, let $k=q+\nu_{\lambda,\alpha}-1$, the non-zero quotient $\sD$-module
\[\sD_X  m f^{-\alpha} / \sD_X m f^{-\alpha} \cap W_{k} (\cM  f^{-\alpha}) \cong \sD_X  m f^{-\alpha}+W_{k} (\cM  f^{-\alpha})/W_{k} (\cM  f^{-\alpha}) \subseteq \gr^W_{k+1}( \cM f^{-\alpha})\]
also has a unique maximal submodule. Hence, the desired statement follows from the semisimplicity of $\gr^W_{k+1}( \cM \cdot f^{-\alpha})$. \end{proof}

\begin{prop}\label{prop:hodgeVgen}
    Assume that $\cS$ underlies a pure Hodge module, $\alpha \in \Q_{>0}$. Let $r_\lambda$ be the largest integer root of $(s+1)b_\lambda(s)$, and consider the set $\Pi_k := \{\ell > r_\lambda \mid (F_{k-\ell} \cS)_{\lambda+\ell \sigma} \neq 0\} \subset \Z_{\geq 0}$. Then
    \[(V^{\alpha}\iota_{+}\cM\cap F_{k+1}\iota_{+}\cM)_{\lambda} \, = \, (p_{\lambda,\alpha}(s)) \bigcap  \underset{\ell \in \Pi_k}{\mathrm{span}_\C} \{[s-\ell+1]_\ell\} \, \otimes \cM_{\lambda} f^s.\]
\end{prop}

\begin{proof}
    Since $\alpha>0$, we have $V^{\alpha}\iota_{+}\cS=V^{\alpha}\iota_{+}\cM$. The strictness implies that
\begin{equation}
    (V^{\alpha}\iota_{+}\cM\cap F_{k+1}\iota_{+}\cM)_{\lambda} \, = \,(V^{\alpha}\iota_{+}\cM\cap F_{k+1}\iota_{+}\cS)_{\lambda}.
\end{equation}
    Under the isomorphism from \eqref{eqn: inverse of Malgrange isomorphism}, one has
    \[ F_{k+1}\iota_{+}\cS\colonequals \bigoplus_{\ell=0}^\infty F_{k-\ell}\cS\otimes \d_t^{\ell}\xrightarrow{\sim}\bigoplus_{\ell=0}^\infty F_{k-\ell}\cS\otimes_{\cO_X} \cO_X(\ell D)\cdot q_{\ell}(s)f^s,\]
where $q_{\ell}(s)\colonequals \prod_{j=0}^{\ell-1}(s-j)$. It follows that
\begin{equation*}\label{eqn: G decomposition of FkiotaM}
(F_{k+1}\iota_{+}\cS)_{\lambda}=\bigoplus_{\ell \geq 0,\,\, (F_{k-\ell}\cS\otimes \cO_X(\ell D))_{\lambda}\neq 0} q_{\ell}(s)\cdot \cM_{\lambda}f^s.
\end{equation*}

Note that by Corollary \ref{cor:sinm}, the condition $(F_{k-\ell}\cS\otimes \cO_X(\ell D))_{\lambda}\neq 0$ is equivalent to $\ell \in \Pi_k$. The conclusion now follows from Theorem \ref{thm: V filtration general G}.
\end{proof}

\begin{remark}
    The finite sets $\Pi_k$ satisfy the following basic properties for $k \in \Z$:
    \begin{enumerate}
    \item $\bigcup_{k=-\infty}^{\infty} \Pi_k = r_\lambda+\Z_{> 0}$,  and if $k \ll 0$ then $\Pi_k = \emptyset$;
    \item $\Pi_k \subset \Pi_{k+1}$;
     \item $\Pi_k+1 \subset \Pi_{k+1}$; 
     \item $\Pi_k - 1 \subset \Pi_{k+d-1}\cup\{r_\lambda\}$.
    \end{enumerate}
    These are the respective consequences of the Hodge filtration being an exhaustive, increasing filtration of $\cO_X$-modules, satisfying Griffiths transversality.
\end{remark}

\begin{thm}\label{thm:FS}
    Let $\alpha \in \Q_{>0}$. We have $(F_k (\cM \cdot f^{-\alpha}))_{\lambda-\alpha\sigma} \neq 0$ if and only if there exists $h(s) \in \C[s]$ with $h(-\alpha)\neq 0$ such that
    \[h(s) \cdot p_{\lambda,\alpha}(s) \in   \underset{\ell \in \Pi_k}{\mathrm{span}_\C} \{[s-\ell+1]_\ell\}.\]
\end{thm}

\begin{proof}
    Assume that $(F_k (\cM \cdot f^{-\alpha}))_{\lambda-\alpha\sigma} \neq 0$. Then by Proposition \ref{prop:hodgeVgen}, (\ref{eqn: Hodge filtration on M via V filtration}), and the (twisted) equivariance of $\mathrm{ev}_{s=-\alpha}$, there exists \[ g(s) \in \, (p_{\lambda,\alpha}(s)) \bigcap  \underset{\ell \in \Pi_k}{\mathrm{span}_\C} \{[s-\ell+1]_\ell\}\]
with $g(-\alpha)\neq 0$. Since $p_{\lambda,\alpha}(-\alpha)\neq 0$, there exists $h(s) \in \C[s]$ with $h(-\alpha)\neq 0$ such that $g(s) = h(s) \cdot p_{\lambda,\alpha}(s)$.

    Conversely, if such $h(s)$ exists, then $h(s) \cM_\lambda f^s \subseteq (V^{\alpha}\iota_{+}\cM\cap F_{k+1}\iota_{+}\cM)_{\lambda} $ by Proposition \ref{prop:hodgeVgen}, and under (\ref{eqn: Hodge filtration on M via V filtration}) we obtain $(F_k (\cM \cdot f^{-\alpha}))_{\lambda-\alpha \sigma} \neq 0$.
\end{proof}
The above result completely determines the isotypic components of the Hodge filtration on $\cM \cdot f^{-\alpha}$, with the knowledge of $\Pi_k$, i.e. isotypic components of Hodge filtration on $\cS$. Note that the condition $\alpha >0$ is not restrictive, as one has $F_k (\cM\cdot f^{-\alpha}) = F_k (\cM\cdot f^{-(\alpha+r)})$, for any $r \in \Z$.

\begin{corollary}\label{cor:Fpart} Let $\alpha \in \Q_{>0}$. We have the following:
    \begin{enumerate}
        \item If $\max \Pi_k < \deg p_{\lambda,\alpha}(s)$, then $(F_k (\cM \cdot f^{-\alpha}))_{\lambda-\alpha\sigma} =0$.
        \item If $\{r_\lambda+1, \dots, \deg p_{\lambda,\alpha}(s)\} \subset \Pi_k$, then $(F_k (\cM \cdot f^{-\alpha}))_{\lambda-\alpha\sigma} = \cM_\lambda \cdot f^{-\alpha}$. 
    \end{enumerate}
\end{corollary}

\begin{proof}
The first case immediately follows from Theorem \ref{thm:FS}.

Now consider case (2). The assumption then implies that 
\[([s-r_{\lambda}]_{r_\lambda+1})_{\leq \deg p_{\lambda, \alpha}} \subset \underset{i \in \Pi_k}{\mathrm{span}_\C} \{[s-i+1]_i\}.\]
Since $[s-r_{\lambda}]_{r_\lambda+1} \mid p_{\lambda, \alpha}(s)$ by \eqref{eqn: plambdaalpha in general}, we get $p_{\lambda,\alpha}(s) \in \underset{i \in \Pi_k}{\mathrm{span}_\C} \{[s-i+1]_i\}$, which proves the claim by Theorem \ref{thm:FS}.
\end{proof}

We will not pursue describing in general $\Pi_k$ here, and leave it for future work. However, in the most important case when $\cS=\cO_X$, this is easy to describe, where the Hodge filtration is given by $F_k(\cO_X)=\cO_X$ if $k\geq 0$ and $F_{-1}\cO_X=0$. Clearly, in this case 
\[\Pi_k = \{r_\lambda+1, \dots, k\}.\]

\begin{proof}[Proof of Theorem \ref{thm:hodgeisot}] For $\alpha \in \Q_{>0}$, this follows readily from Corollary \ref{cor:Fpart}. Now assume $\alpha\leq 0$, choose $r\in \Z$ such that $r+\alpha>0$. For $m\in \cM_{\lambda}$, the condition $m\cdot f^{-\alpha}\in F_{k}(\cM\cdot f^{-\alpha})$ is equivalent to:
\begin{align*}
    (mf^r)\cdot f^{-(\alpha+r)} \in F_k(\cM\cdot f^{-(\alpha+r)})\Longleftrightarrow \deg p_{mf^r,\alpha+r}(s)\leq k \Longleftrightarrow \deg p_{\lambda+r\sigma,\alpha+r}(s)\leq k.
\end{align*}
Since $p_{\lambda+r\sigma,\alpha+r}(s)=p_{\lambda,\alpha}(s+1)$, this is equivalent to $\deg p_{\lambda,\alpha}(s)\leq k$.
\end{proof}

\begin{proof}[Proof of Theorem \ref{thm: V cap F any alpha general G}]
Recall $\cM=(\cO_X)_f$. If $\alpha>0$, it follows directly from Proposition \ref{prop:hodgeVgen}. For general $\alpha$, note that under the isomorphism from \eqref{eqn: inverse of Malgrange isomorphism}, one has 
    \[ F_{k+1}\iota_{+}\cM\colonequals \bigoplus_{i=0}^k F_{k-i}\cM\otimes \d_t^{i}\xrightarrow{\sim}\bigoplus_{i=0}^k F_{k-i}\cM\otimes_{\cO_X} \cO_X(i D)\cdot q_{i}(s)f^s,\]
where $q_{i}(s)\colonequals \prod_{j=0}^{i-1}(s-j)=[s-(i-1)]_{i}$ and $q_0(s)=1$. It follows that
\begin{equation*}
(F_{k+1}\iota_{+}\cM)_{\lambda}=\bigoplus_{0\leq i \leq k,\, (F_{k-i}\cM\otimes \cO_X(i D))_{\lambda}\neq 0} q_{i}(s)\cdot \cM_{\lambda}f^s.
\end{equation*}
By \cite[Proposition 13.1]{MP16}, if $i<k$, then $f\cdot F_{k-i}\cM\subseteq F_{k-i-1}\cM$, so
\begin{align*}
F_{k-i}\cM\otimes \cO_X(i D)\subseteq F_{k-i-1}\cM\otimes \cO_X((i+1))D).
\end{align*}
Suppose there exists a minimal integer $\ell\in [0,k]$ such that $(F_{k-\ell}\cM\otimes \cO_X(\ell D))_{\lambda}\neq 0$ (which is equivalent to $\deg p_{\lambda,-\ell}(s)=\deg p_{\lambda+\ell \sigma,0}(s)\leq k-\ell$ by Theorem \ref{thm:hodgeisot}). Then $(F_{k+1}\iota_{+}\cM)_{\lambda}\neq 0$ and
\[ (F_{k+1}\iota_{+}\cM)_{\lambda}=\mathrm{span}_{\C}\{q_{\ell}(s),q_{\ell+1}(s),\ldots, q_{k}(s)\}\cdot \cM_{\lambda}f^s=(q_{\ell}(s))_{\leq k}\cdot \cM_{\lambda} f^s.\]
The formula for $\alpha\in \Q$ now follows from Theorem \ref{thm: V filtration general G}. 
\end{proof}

\begin{proof}[Proof of Corollary \ref{corollary: isotypic component of Hodge and higher}]
    From \eqref{eqn: definition of Hodge ideal}, we have $I_k(\alpha D)\otimes \cO_X(kD)=F_k(\cM\cdot f^{-\alpha})\cdot f^{\alpha}$ and hence Theorem \ref{thm:hodgeisot} implies that $
    (I_k(\alpha D)\otimes\cO_X(kD))_{\lambda}\cong \cM_{\lambda}$ if $\deg p_{\lambda,\alpha}(s)\leq k$ and vanishes otherwise. On the other hand, the definition \eqref{eqn: definition of higher multiplier ideals} and Theorem \ref{thm: V cap F any alpha general G} imply that
    \[ (\tilde{I}_k(\alpha D)\otimes \cO_X(kD)_{\lambda} \cong (p_{\lambda,\alpha}(s))_{=k}\cdot \cM_{\lambda}f^s.\] This completes the proof.
\end{proof}

\subsection{Actions of $(t,\d_t)$ and $(f,\d f)$}
\begin{corollary}\label{corollary: the action of t and dt on V}
Let $t: (V^\alpha \iota_+ \cM)_{\lambda} \to (V^{\alpha+1} \iota_+\cM)_{\lambda+\sigma}$ %var 
and $\partial_t: (V^\alpha \iota_+ \cM)_{\lambda} \to (V^{\alpha-1} \iota_+\cM)_{\lambda-\sigma}$. %can
Define $m_{\lambda+\sigma}=f \cdot m_{\lambda}$ and $m_{\lambda-\sigma}=-m_{\lambda}/f$. Then
    \[t \cdot \left(p_{\lambda,\alpha}(s)m_{\lambda} f^s\right) = p_{\lambda+\sigma,\alpha+1}(s) m_{\lambda+\sigma} f^s, \quad \partial_t \cdot \left(p_{\lambda,\alpha}(s) m_{\lambda} f^s\right) = s \cdot p_{\lambda-\sigma,\alpha-1}(s) m_{\lambda-\sigma} f^s.\]
\end{corollary}

\begin{proof}
Since $b_{\lambda+\sigma}(s)= b_{\lambda}(s+1)$, we have
\begin{equation}\label{eqn: plambda+sigma}
    p_{\lambda+\sigma,\alpha+1}(s)=\prod_{i=1}^d [s+\lambda_i+1]_{\lceil \alpha+1-(\lambda_i+1)\rceil}=p_{\lambda,\alpha}(s+1).
\end{equation}
This proves the first formula. The second follows from the identity $\d_t=-s\cdot t^{-1}$. 
\end{proof}

We now describe the linear maps corresponding to multiplication by $f$ and $\partial f$ on the $\lambda$-isotypic components of $\gr^\alpha_V \iota_+ \cM$. These maps encode additional information about the $\sD$-module structure of $\gr^\alpha_V \iota_+ \cM$, e.g. the extensions between its composition factors. In certain cases, they even determine the entire $\sD$-module structure \cite{nang} (see also \cite[Section 6]{LW19}). Since $f$ and $\d_f$ are semi-invariant, Schur's lemma and Remark \ref{rem:weighgrVlambda} imply that their actions are given by the matrices of the form $A \otimes \textrm{Id}_{\dim \cM_\lambda}$, for some matrices $A$. Thus, it suffices to describe these matrices $A$, which we do in the basis $1, s+\alpha, (s+\alpha)^2, \dots$ of $\C[s]$. 

\begin{corollary}
    Fix $\alpha \in \Q$ and $\ell \in \Z$ such that $\nu_{\lambda,\alpha} > - \ell$. Let $\rho:=\mathrm{mult}_{s=-\alpha} b_\lambda(s)$, $\nu:= \min\{\lfloor \frac{\nu_{\lambda,\alpha}+\ell+1}{2} \rfloor, \nu_{\lambda, \alpha}\}$, and $\mu:= \min\{\lfloor \frac{\nu_{\lambda, \alpha}-\rho+\ell+1}{2} \rfloor, \nu_{\lambda, \alpha}-\rho\}$. Write \[b_\lambda(s)/(s+\alpha)^\rho = \sum_{i \geq 0} c_i \cdot (s+\alpha)^i\]
    Then the matrices of 
    \[f: (W_\ell \gr_V^\alpha \iota_+ \cM)_{\lambda-\sigma} \to (W_\ell \gr_V^\alpha \iota_+ \cM)_{\lambda},\quad \partial f: (W_\ell \gr_V^\alpha \iota_+ \cM)_{\lambda+\sigma} \to (W_\ell \gr_V^\alpha \iota_+ \cM)_{\lambda},\]
    are given by $[\mathrm{Id}_\nu \,\,\, 0]$ and $[0 \,\,\, C]^{T}$, respectively, where $C$ is the $\mu \times (\nu-\rho)$ matrix
    \[C=\begin{pmatrix}
        c_0 & c_1 & c_2 & \dots \\
        0 & c_0 & c_1 & \dots \\
        0 & 0 & c_0 & \dots \\
        \vdots & \vdots & \vdots & \ddots
    \end{pmatrix}.\]
\end{corollary}

\begin{proof}
    Follows directly from Remark \ref{rem:weighgrVlambda}.
\end{proof}
\begin{remark}
     Note that if $\ell \geq v_{\lambda,\alpha}-1$, then $C$ is invertible. Moreover, when $\rho=0$, the map $\partial f$ is invertible. Conversely, if $\rho\neq 0$, then both $f \circ \partial f$ and $\partial f \circ f$ act nilpotently on $(W_\ell \gr_V^\alpha \iota_+ \cM)_{\lambda}$.
\end{remark}

\section{Multiplicity-free spaces}\label{sec: V filtration on multiplicity free spaces}

Assume $X$ is a multiplicity-free space. Since $G$ acts on $X$ with finitely many orbits, any coherent equivariant $\sD_X$-module is regular holonomic and quasi-unipotent (see \cite[Section 11.6]{HTT}; the quasi-unipotence follows because simples come from local systems with finite monodromy, and quasi-unipotence is preserved under extensions). Thus, any equivariant $\sD_X$-module admits a $V$-filtration along any semi-invariant function. Since the functor $V^\alpha(-)$ is exact due to the strict compatibility property, it suffices to study simple equivariant $\sD_X$-modules $\cS$. If $\cS$ is supported in the hypersurface $f^{-1}(0)$, the $V$-filtration is easily described \cite[Lemme 3.1.3]{Saito88}. Therefore, we assume $\mathrm{supp}(\cS)\not\subset f^{-1}(0)$ and set $\cM:= \cS_f$.

\subsection{Decomposition of $V$-filtrations}

Recall that $\Lambda(\mathcal{M})$ is the set of integral dominant weights $\lambda$ for which $\cM_\lambda \neq 0$. By Lemma \ref{lem:multfree}, $\cM_\lambda = U_\lambda$ for any $\lambda$, yielding the decomposition into irreducibles
\begin{equation*}\label{eq:multfree}
    \cM = \bigoplus_{\lambda \in \Lambda(\cM)} \, U_{\lambda}.
\end{equation*}
Let $U(\mathfrak{g})$ be the universal enveloping algebra of the Lie algebra $\mathfrak{g}$ of $G$. Then, extending scalars, we have a decomposition as $U(\mathfrak{g})[s]$-modules:
\begin{equation}\label{eq:multfreefs}
    \cM[s]f^s = \bigoplus_{\lambda \in \Lambda(\cM)} \, U_\lambda[s]  f^s,
\end{equation}
where $U_\lambda[s]:=U_\lambda\otimes \C[s]$ is the $\lambda$-isotypic component of $\cM[s]$, and $f^s$ has $\lie$-weight $s \cdot \sigma$.

We may also view \eqref{eq:multfreefs} as a decomposition of $(\C[s],G)$-modules, treating $f^s$ as a $G$-invariant element. Since the $V$-filtration is both $G$- and $\C[s]$-stable (cf. Lemma \ref{lemma: G and s compatible on V filtration}), and these actions commute, we obtain a $(\C[s], G)$-module decomposition:
\begin{equation}\label{eq:vfiltp}
V^\alpha \iota_+ \cM = \bigoplus_{\lambda \in \Lambda(\cM)} (p_{\lambda, \alpha}(s)) \otimes \cM_{\lambda} f^s, \quad \textrm{for any $\alpha\in \Q$},
\end{equation}
where $p_{\lambda,\alpha}(s) =p_{m,\alpha}(s)$ for any $m\in \cM_{\lambda}$ and is determined by Theorem \ref{thm: V filtration general G}.

\subsection{More on $b$-functions when $\cS=\cO_X$}\label{sec:moreb}
Let $\cM=(\cO_X)_f$. Let $B\subseteq G$ be a Borel subgroup, and consider the ring of highest weight vectors (i.e. $B$-semi-invariants), which forms a polynomial ring \cite[\S 4.]{saki} generated by irreducible highest weight polynomials (up to scalar) $h_i$, so that $\cO_X^{[B,B]} = \C[h_1, \dots, h_k]$. Let $\lambda^i$ denote the highest weight of $h_i$. Then $\Lambda(\cO_X)$ is freely generated as an additive semigroup by $\lambda^1,\dots, \lambda^k$. 

The following shows that computing the $b_\lambda(s)$ for all $\lambda \in \Lambda(\cM)$ is a finite task (suffices to work with $\lambda\in \Lambda(\cO_X)$). Let the $b$-function of $f$ factor as $\prod_{s=1}^d (s+r_i)$, with $r_i \in \Q_{>0}$. 

\begin{thm}\label{thm: computation of blambda} Assume the complement of the dense $G$-orbit in $X$ is a hypersurface. Then there exists constants $c_{ij}\in \Q_{\geq 0}$, with $c_{ij} \leq \deg h_j$, such that
\[b_{\sum_{i=1}^k a_i \lambda^i}(s) = \prod_{i=1}^d (s+ r_i + \sum_{j=1}^k c_{ij} a_j), \quad \forall (a_1, a_2, \dots, a_k) \in \Z_{\geq 0}^k.\]
\end{thm}

\begin{proof} 
Since $G$ is reductive, the assumption is equivalent to $(G,X)$ being a regular prehomogeneous space, i.e. the existence of a $G$-semi-invariant function on $X$ whose Hessian determinant is nonzero \cite[\S 4.]{saki}. Then $(B,X)$ is also a regular prehomogeneous space, and Sato's theorem on functions satisfying the cocycle condition applies \cite[\S 4. Theorem 2.]{sato},  granting the existence of a factorization
\[b_{\sum_{i=1}^k a_i \lambda^i}(s) = \prod_{i=1}^d (s+ r'_i + \sum_{j=1}^k c_{ij} a_j)\]
for all $(a_1, a_2, \dots, a_k) \in \Z_{\geq 0}^k$. Setting $a_i=0$ for all $i$ shows $r_i' = r_i$.

By Lemma \ref{lem:bdiv}, $b_{a\lambda^i}(s)$ must have negative rational roots for all $a\in \Z_{\geq 0}$, and taking $a \gg 0$ shows $c_{ji} \in \Q_{\geq 0}$ with $c_{ji} \leq  \deg h_i$ for all $j=1,\dots, d$.
\end{proof}

Furthermore, under the same assumption, $b$-functions exhibits a symmetry \cite[Theorem 3.13]{holoprehom}. For simplicity, we state it when the hypersurface is irreducible \cite[Corollary 3.14]{holoprehom}, which will be useful in Proposition \ref{prop: b function E6}.

\begin{prop}\label{prop:bfunsymmetry}
    Assume $f$ is irreducible and $X \setminus \{f=0\}$ is a dense $G$-orbit. Let $n=\dim X$. Then for any $\lambda \in \Lambda(\cM)$, we have $\lambda^* \in \Lambda(\cM)$, and
    \[b_{\lambda}(s) = (-1)^d \cdot b_{\lambda^*}(-s-\frac{n}{d}-1).\]
\end{prop}

\begin{proof}
As in the proof of \cite[Lemma 3.9]{holoprehom}, to apply \cite[Corollary 3.14]{holoprehom}, it suffices to verify that for $p\gg 0$, the shifted weight $\lambda'=\lambda + p \sigma \in \Lambda(\cO_X)$ is a witness representation of $\cO_X$, meaning that for any simple equivariant $\sD_X$-module $\cN \neq \cO_X$ with $\mathrm{supp} \, \cN = X$, we have $\cN_{\lambda'}=0$. This follows from \cite[Corollary 3.23]{LW19}.
\end{proof}

\begin{remark}\label{rem:cijnilp}
    Proposition \ref{prop:nilp} imposes further constraints on the $c_{ij}$: write $b_\lambda(s) = \prod_{i=1}^d (s+\lambda_i)$, then for any $\alpha \in \Q$ the number
 $\# \{ i \mid \lambda_i \in \alpha+ \mathbb{Z}  \}$ is independent of $\lambda \in \Lambda(\cM)$.
\end{remark}

\begin{remark}
See also  \cite{multfreeactions} for related results on eigenvalues of invariant operators for multiplicity-free actions, and references therein.
\end{remark}

\subsection{$G$-decomposition of various filtrations}
\begin{proof}[Proof of Theorem \ref{thm: V filtration on multiplicity free space via representation data}]
It follows from Theorem \ref{thm: V filtration general G} and \eqref{eq:vfiltp}. 
\end{proof}
\begin{remark}
    Corollary \ref{cor: ValphaiotaS} (see there for the notation $r_\lambda$) together with \eqref{eq:vfiltp} implies that
\begin{equation*}\label{eqn: V0S}
V^\alpha \iota_{+} \mathcal{S} =\bigoplus_{\lambda \in \Lambda(\mathcal{\cM})} \left(\mathrm{lcm}\{[s-r_{\lambda}]_{r_{\lambda}+1}, p_{\lambda,\alpha}(s)\}\right) \otimes \cM_\lambda f^s.
\end{equation*}
\end{remark}

\begin{proof}[Proof of Corollary \ref{corollary: G decomposition of grV}]
From \eqref{eqn: plambdaalpha in general} and \eqref{eqn: numalpha}, we have $p_{\lambda,\alpha+\varepsilon}(s) = p_{\lambda,\alpha}(s) \cdot (s+\alpha)^{\nu_{\lambda,\alpha}} $
for small enough $\varepsilon>0$. Hence the $(\C[s], G)$-module decomposition of $\gr_V^\alpha \iota_+ \cM$ follows immediately from Theorem \ref{thm: V filtration on multiplicity free space via representation data}. To compute the weight filtration, observe:
\[\mathrm{Im} N^\ell = (s+\alpha)^\ell \cdot \gr_V^{\alpha}{\iota_+} \cM = \bigoplus_{\lambda \in \Lambda(\cM),\nu_{\lambda, \alpha}>\ell} \C[s]/((s+\alpha)^{\nu_{\lambda, \alpha}-\ell}) \otimes \cM_{\lambda}f^s,\]
\[\ker N^\ell = \bigoplus_{\lambda \in \Lambda(\cM)} ((s+\alpha)^{\max\{\nu_{\lambda, \alpha}-\ell, 0\}})/((s+\alpha)^{\nu_{\lambda, \alpha}}) \otimes \cM_{\lambda}f^s = \bigoplus_{\lambda \in \Lambda(\cM)} \C[s]/((s+\alpha)^{\min\{\ell, \nu_{\lambda, \alpha}\}}) \otimes \cM_{\lambda}f^s\]
Combining these with the convolution formula \eqref{eqn: convolution formula} gives \eqref{eqn: G decomposition of grV}.

The semisimplicity of $\gr^{W(N)}_{\ell}\gr^{\alpha}_V$ follows from \cite[Theorem 19.4.2]{Mochizuki}. Recall from \S \ref{sec: Gequivariant diff} that the differential operators yielding the $G$-action in \eqref{eqn: G decomposition of grV} are of the form $\xi - s \cdot d\sigma(\xi) \in \sD_X[s]$, for $\xi \in \mathfrak{g}$. To recover the natural $\mathfrak{g}$-action via vector fields, we twist by $\alpha \cdot d\sigma(\xi)$ (note $\xi - s \cdot d\sigma(\xi)$ acts as $\xi + \alpha \cdot d\sigma(\xi) \in \sD_X$). The remaining parts follow.
\end{proof}
\begin{proof}[Proof of Corollary \ref{thm: Gdecomposition of weight filtration}] By \eqref{eqn: associated graded of weight} and Corollary \ref{corollary: G decomposition of grV}, we have a $\lie$-decomposition 
\[ \gr^W_{\ell+q}(\cM\cdot f^{-\alpha}) \cong \frac{\gr^{W(N)}_{\ell-1}\gr^{\alpha}_V\iota_{+}\cM}{N\cdot \gr^{W(N)}_{\ell+1}\gr^{\alpha}_V\iota_{+}\cM}=\bigoplus_{\nu_{\lambda,\alpha}=\ell}U_{\lambda-\alpha \sigma}, \quad \forall \ell>0.\]
Using \eqref{eqn: definition of Salpha}, this proves \eqref{eqn: G decomposition of Well M}. The final claim follows since $\nu_{\lambda,\alpha}\leq d$ (see \eqref{eqn: numalpha}).
\end{proof}
\begin{corollary}\label{cor:samenumroot}
    Let $\cN$ be a simple composition factor of the $\sD$-module $\cM \cdot f^{-\alpha}$. Then for any $\lambda, \mu \in \Lambda(\cM)$ such that $\cN_{\lambda - \alpha \sigma} \neq 0 \neq \cN_{\mu - \alpha \sigma} $, we have $\nu_{\lambda,\alpha} = \nu_{\mu, \alpha}$.
\end{corollary}

\begin{proof}[Proof of Corollary \ref{cor: G decomposition of Hodge filtration}]
This follows from Theorem \ref{thm:hodgeisot}.
\end{proof}

\begin{proof}[Proof of Corollary \ref{corollary: G decomposition of Hodge ideals}]

Since $\cO_X$ is multiplicity-free, Corollary \ref{corollary: isotypic component of Hodge and higher} implies that $I_k(\alpha D)=\tilde{I}_{k}(\alpha D)$ as ideal sheaves.
\end{proof}
\section{Examples}

We now apply our general results to some multiplicity-free representations (classified in \cite{class1}, \cite{class2}), where our results are the strongest; and some prehomogeneous spaces.

\subsection{Generic determinant}\label{sec: generic determinant}
Let $X= (\C^n)^* \otimes (\C^n)^*$, identified with the space of $n\times n$ matrices, with $G=\GL_n(\C) \times \GL_n(\C)$ acting naturally by row and column operations. The irreducible semi-invariant function $f$ is the determinant, with weight $\sigma = (\det,\det)$, where $\det=(1^n)$ and $\deg f = n$. The $G$-decomposition of $\cM=(\cO_X)_f$ is given by the Cauchy formula (see \cite[Corollary 2.3.3]{weymanbook})
\begin{equation*}\label{eq:cauchy}
\cM= \bigoplus_{\underline{p}} U_{\underline{p}}=\bigoplus_{\underline{p}\in \Z^n_{\mathrm{dom}}} S_{\underline{p}} \C^n \otimes S_{\underline{p}} \C^n,
\end{equation*}
where  $\Z^n_{\mathrm{dom}}=\{\underline{p}\in \Z^n \mid p_1\geq p_2\geq \ldots \geq p_n\}$,
and $S_{\underline{p}}\C^n$ denotes the irreducible $\GL_n$-representation of highest weight $\underline{p}$. For any non-zero $m_{\underline{p}} \in S_{\underline{p}} \C^n \otimes S_{\underline{p}} \C^n$, let $b_{\underline{p}}(s)$ denote its $b$-function with respect to $f$. Then Proposition \ref{prop:bfunirred} gives
\[  \partial f \cdot m_{\underline{p}} f^{s+1} = b_{\underline{p}}(s) \cdot m_{\underline{p}} f^s.\]

\begin{prop}\label{prop: bfunction iota determinant}
For $\underline{p}=(p_1,\dots, p_n)\in \Z^n_{\mathrm{dom}}$, we have $b_{\underline{p}}(s) =\prod_{i=1}^n(s+1+p_i+n-i)$.
\end{prop}

\begin{proof}
We sketch the proof, as it is analogous to the classical determination of the Bernstein--Sato polynomial of $f$ via the Capelli identity. After multiplying by a suitable power of $f$, we may assume that $p_n\geq 0$ so that $m_{\underline{p}}$ can be taken as a suitable product of principal minors of the generic $n \times n$ matrix. Applying the Capelli identity expresses $\partial f \cdot f$ as the determinant of Capelli operators \cite[Section 3.4]{procesi}, which are easy to evaluate on $m_{\underline{p}} f^s$.
\end{proof}

\begin{remark}
    The $I_k(\alpha D)$ in Corollary \ref{for: Hodge ideals and higher multiplier ideals for determinants} can be computed directly using \cite[Theorem 7.1]{dCEP80}. Specifically, for $\alpha\in (0,1]$ and $\lambda=(p_1,\ldots,p_n)\in \Z^n_{\mathrm{dom}}$, we have: 
  \[  \deg p_{\lambda,\alpha+k}(s)\leq k\Longleftrightarrow\sum_{i=t}^n p_i \geq (n-t)(k-1)-\binom{n-t}{2}, \quad \forall 1\leq t\leq n.\]
  Moreover, this recovers the primary decomposition of $I_k(D)$ in \cite[Theorem 1.1]{PerlmanRaicu}.
\end{remark}
\begin{remark}\label{remark: nearby cycle perverse sheaf for generic determinant}
We can also obtain the following computation of nearby and vanishing cycle mixed Hodge modules. Denote by $\IC_{O_j}$ the intersection cohomology module associated to $O_j$, the loci of rank $j$ matrices. Then $\gr^{\alpha}_V\iota_{+}\cO_X=0$ if $\alpha\not\in \Z$ and 
\[ \gr^{W(N)}_{\ell}\gr^1_V\iota_{+}\cO_X=\bigoplus_{j\geq 0}\IC_{O_{n-\ell-1-2j}}, \quad \gr^{W(N)}_{\ell}\gr^0_V\iota_{+}\cO_X=\gr^{W(N)}_{\ell+1}\gr^1_V\iota_{+}\cO_X,\]
if $\ell\geq 0$, and $\gr^{W(N)}_{\ell}\gr^1_V\iota_{+}\cO_X\cong \gr^{W(N)}_{-\ell}\gr^1_V\iota_{+}\cO_X$ if $\ell<0$. In particular $\gr^{W(N)}_{\ell}\gr^1_V\iota_{+}\cO_X$ is nonzero if and only if $-(n-1)\leq \ell\leq n-1$. This recovers \cite[Theorem 4.12]{BG}.  To obtain the claim, by Corollary \ref{corollary: G decomposition of grV}, we have a $G$-decomposition
\[ \gr^{W(N)}_{\ell}\gr^1_V\iota_{+}\cO_X=\bigoplus_{\underline{p}\in \Z^n_{\mathrm{dom}}, \, \nu_{\underline{p},1}=\ell+1+2j,\, j\in \Z_{\geq 0}} U_{\underline{p}}.\]
Proposition \ref{prop: bfunction iota determinant} and \eqref{eqn: numalpha} imply that
\[ \nu_{\underline{p},1}=\#\{ i \mid 1-(p_i+n-i+1)\in \Z_{\geq 0}\}=\#\{ i \mid p_i\leq i-n\}.\]
Then one concludes via the $G$-decomposition of $\IC_{O_j}$, e.g. in \cite{characters} or \cite[(2.7)]{PerlmanRaicu}.

Since $\IC^H_{O_j}$ is the only Hodge module with the underlying $\sD$-module $\IC_{O_p}$ up to a Tate twist (\cite[\S 2.3]{PerlmanRaicu}), one can lift the equalities as mixed Hodge modules. Furthermore, one can describe the Hodge and weight filtrations (in particular, the Tate twists) using Corollaries \ref{cor: Hodge filtration on nearby cycle} and \ref{corollary: G decomposition of grV}, which recovers \cite[Theorem 1.4]{PerlmanRaicu}. We leave the details for the interested reader. We give a detailed discussion for the Pfaffian case in Corollary \ref{cor: simple for Phaffian}.
\end{remark}

\subsection{Symmetric determinant}

Let $X = \Sym_2 (\C^n)^*$ be the space of symmetric $n\times n$ matrices, equipped with the natural action of $G=\GL_n(\C)$. The semi-invariant function $f$ is the determinant of a generic symmetric matrix, with weight $\sigma = \det^2 = (2^n)$ and $\deg f = n$. The $G$-decomposition of $\cM=(\cO_X)_f$ is given by (e.g. \cite[Proposition 2.3.8]{weymanbook})
\begin{equation}\label{eq:cauchy symmetric}
\cM= \bigoplus_{\underline{p}} U_{\underline{p}}=\bigoplus_{\underline{p}\in \Z^n_{\mathrm{dom}}, \, p_i \text{ even}} S_{\underline{p}} \C^n.
\end{equation}
As before, denote by $b_{\underline{p}}(s)$ the $b$-function associated to $\underline{p}$.

\begin{prop}\label{prop: bfunction iota determinant symmetric}
For $\underline{p}=(p_1,\dots, p_n)\in \Z^n_{\mathrm{dom}}$ with all $p_i$ even, we have 
\[ b_{\underline{p}}(s) = \prod_{i=1}^n \left(s+1+\frac{p_i+n-i}{2} \right).\]
\end{prop}

\begin{proof}
    The proof is analogous to that of Proposition \ref{prop: bfunction iota determinant}, using the symmetric version of the Capelli identity \cite{howeumeda}. More precisely, consider the Capelli-type identity in $\sD_X$:
    \begin{equation}\label{turnbull}
    f \cdot \partial f = \det(E_{ij} + (n-i)\delta_{ij}),
    \end{equation}
    where $f=\det(x_{ij})$ with variables $x_{ij}=x_{ji}$, $\partial f = \det((1+\delta_{ij})\partial_{ij})$, 
    \[ E_{ij} = \sum_{k=1}^n (1+\delta_{k,j}) x_{ki} \partial_{kj} \, + (n-i)\delta_{ij},\]
    and $\delta_{ij}$ is the Kronecker delta. With a suitable choice of triangular matrices $B \subseteq \GL_n(\C)$, using the notation from Section \ref{sec:moreb} we have $\cO_X^{[B,B]} = \C[h_1, \dots, h_n]$, where $h_k$ denotes the principal $k\times k$ minor $\det([x_{ij}]_{1\leq i, j \leq  k})$ with highest weight $(2^k)$. Thus, the element $h^{\underline{a}}:=\prod_{i=1}^n h_i^{a_i}$ is the highest weight vector of $S_{\underline{p}} \C^n$, with $p_i = 2 \sum_{j=i}^n a_j$.

    By Proposition \ref{prop:bfunirred} and (\ref{turnbull}), we have an equation
    \begin{equation}\label{hwv}
        \det(E_{ij} + (n-i)\delta_{ij}) \cdot h^{\underline{a}} f^{s+1} = (f \cdot \partial f) \cdot h^{\underline{a}} f^{s+1} = c \cdot b_{\underline{p}}(s)  \cdot h^{\underline{a}}f^{s+1},
    \end{equation}
with a non-zero scalar $c$. Since $h^{\underline{a}}f^{s+1}$ is a highest weight vector, we have $E_{ij} \cdot h^{\underline{a}}f^{s+1} =0$, whenever $i<j$, and $E_{ii}  \cdot h^{\underline{a}}f^{s+1} = (2s+2+2 \sum_{j=i}^n a_j) h^{\underline{a}}f^{s+1}$. As  $\det(E_{ij} + (n-i)\delta_{ij})$ is expanded in the left to right order, only the diagonal product can act non-zero, that is,
\[\det(E_{ij} + (n-i)\delta_{ij}) \cdot h^{\underline{a}}f^{s+1} =  \prod_{i=1}^n (E_{ii}+(n-i)) \, \cdot h^{\underline{a}}f^{s+1} = \prod_{i=1}^n (2s+2 + n-i+ 2\sum_{j=i}^{n} a_i) \cdot h^{\underline{a}}f^{s+1}.\]
This, together with (\ref{hwv}) and $p_i = 2 \sum_{j=i}^n a_j$, yields the result.
\end{proof}

\begin{corollary}\label{cor: hodge ideal symmetric}
    Let $k\in \N$ and $\alpha\in (0,1]$. Then $\tilde{I}_k(\alpha D)=I_k(\alpha D)$ has a primary decomposition
   \[\begin{cases}\bigcap_{n-t \textrm{ even}} J_t^{\left(\lceil \frac{1}{2}(n-t)(2k-\frac{n-t}{2}-1)\rceil\right)}\cap \bigcap_{n-t \textrm{ odd}}J_t^{\left(\lceil \frac{1}{2}(n-t)(2k-\frac{n-t}{2}-1)-\frac{1}{4}\rceil\right)}, \quad &\textrm{if $\alpha\in (0,1/2]$},\\
\bigcap_{n-t \textrm{ even}} J_t^{\left(\lceil \frac{1}{2}(n-t)(2k-\frac{n-t}{2})\rceil\right)}\cap \bigcap_{n-t \textrm{ odd}}J_t^{\left(\lceil \frac{1}{2}(n-t)(2k-\frac{n-t}{2})+\frac{1}{4}\rceil\right)}, \quad &\textrm{if $\alpha\in (1/2,1]$}. \end{cases}\]
where $J_t^{(d)}$ denotes the symbolic power of $J_t$, the ideal of functions vanishing to order $d$ along the subvariety of symmetric matrices of rank $<t$ and $J_t^{(d)}=\cO_X$ if $d\leq 0$.
\end{corollary}

\begin{remark}\label{remark: first hodge ideal for symmetric}
    Note that for $0<\epsilon\ll 1$, $\tilde{I}_1\left((\frac{1}{2}+\epsilon)D\right)=J_{n-1}=I_{D_{\mathrm{sing}}}\neq \cO_X$ is the first proper ideal, which is compatible with $\tilde{\alpha}_f=\frac{3}{2}$ (see \cite[Lemma 3.10]{SY23}). Applying \cite[Theorem E]{MP16}, we obtain
    \[ I_k(D)\subseteq J_t^{(q)}, \quad \textrm{where $q=\min\left(n-t,(k+1)(n+1-t)-\binom{n+2-t}{2}\right)$}.\]
     Our result shows that, for large $k$, the optimal exponent is $\lceil \frac{1}{2}(n-t)(2k-\frac{n-t}{2}-1)-\frac{1}{4}\rceil$.
\end{remark}

\begin{proof}
    Denote by $\lambda=(p_1,\ldots,p_n)\in \Z^n_{\mathrm{dom}}$ with $p_i$ even. Using \eqref{eqn: plambdaalpha in general}, Proposition \ref{prop: bfunction iota determinant symmetric} and Corollary \ref{corollary: G decomposition of Hodge ideals}, it suffices to consider those $\lambda$ such that $\deg p_{\lambda,\alpha+k}(s)\leq k$, which is equivalent to
\begin{equation}\label{eqn: degree p leq k symm}
\sum_{i=1}^{n} \max(\lceil(\alpha+k)-(1+\frac{p_i+n-i}{2})\rceil,0)\leq k.
\end{equation}
Since $p_i\geq p_{i+1}$ for each $i$, one has
\[ \lceil(\alpha+k)-(1+\frac{p_i+n-i}{2})\rceil \leq \lceil(\alpha+k)-(1+\frac{p_{i+1}+n-(i+1)}{2})\rceil.\]
Therefore \eqref{eqn: degree p leq k symm} is equivalent to
\begin{equation}\label{eqn: bound on the partial sum symm}
\sum_{i=t}^{n} \lceil(\alpha+k)-(1+\frac{p_i+n-i}{2})\rceil\leq k, \quad \textrm{for all $1\leq t\leq n$}.\end{equation}
Note that one has 
\[ \lceil(\alpha+k)-(1+\frac{p_i+n-i}{2})\rceil=\begin{cases}k-\frac{p_i+n-i}{2}, &\textrm{if $n-i$ even, $\alpha\in (0,1]$},\\
k-\frac{p_i+n-i}{2}-\frac{1}{2}, &\textrm{if $n-i$ odd, $\alpha\in (0,1/2]$},\\
k-\frac{p_i+n-i}{2}+\frac{1}{2},&\textrm{if $n-i$ odd, $\alpha\in (1/2,1]$}.
\end{cases}\]

Using this, let us compute \eqref{eqn: bound on the partial sum symm}. First suppose $\alpha\in (0,1/2]$, then
\[ \sum_{i=t}^{n} \lceil(\alpha+k)-(1+\frac{p_i+n-i}{2})\rceil=\sum_{i=t}^n (k-\frac{p_i+n-i}{2}) -\frac{1}{2}\cdot \#,\]
where $\#$ is the number of $i$ between $t$ and $n$ such that $n-i$ is odd. If $n-t$ is even, then $\#=(n-t)/2$; if $n-t$ is odd, then $\#=(n-t+1)/2$. Hence \eqref{eqn: bound on the partial sum symm} is equivalent to
    \[\begin{cases}-\frac{n-t}{4}+\sum_{i=t}^{n} (k-\frac{p_i+n-i}{2})\leq k, &\textrm{if $n-t$ even},\\
    -\frac{n-t+1}{4}+\sum_{i=t}^{n} (k-\frac{p_i+n-i}{2})\leq k, &\textrm{if $n-t$ odd}.
    \end{cases}\]
Notice that
\begin{align*} \sum_{i=t}^{n} (k-\frac{p_i+n-i}{2})=k(n-t+1)-\frac{1}{2}\sum_{i=t}^n p_i-\frac{(n-t)(n-t+1)}{4}
\end{align*}
So if $n-t$ is even, then  \eqref{eqn: bound on the partial sum symm}  is equivalent to 
\[ \sum_{i=t}^n p_i\geq 2\cdot \left(-\frac{n-t}{4}+k(n-t)-\frac{(n-t)(n-t+1)}{4}\right)=(n-t)(2k-\frac{n-t}{2}-1).\]
Similarly, if $n-t$ is odd, we have
\begin{align*}\sum_{i=t}^n p_i&\geq 2\cdot \left(-\frac{n-t}{4}+k(n-t)-\frac{(n-t)(n-t+1)}{4}-\frac{1}{4}\right)=(n-t)(2k-\frac{n-t}{2}-1)-\frac{1}{2}.\end{align*}
Hence we conclude that  \eqref{eqn: bound on the partial sum symm}  is equivalent to 
    \[\sum_{i=t}^n p_i\geq \begin{cases} (n-t)(2k-\frac{n-t}{2}-1), \quad &\textrm{if $n-t$ even},\\
    (n-t)(2k-\frac{n-t}{2}-1)-\frac{1}{2}, \quad &\textrm{if $n-t$ odd}.
    \end{cases}\] 

If $\alpha\in (1/2,1]$, a similar computation shows that \eqref{eqn: bound on the partial sum symm}  is equivalent to 
    \[\sum_{i=t}^n p_i\geq \begin{cases}(n-t)(2k-\frac{n-t}{2}), \quad &\textrm{$n-t$ even},\\
    (n-t)(2k-\frac{n-t}{2})+\frac{1}{2}, \quad &\textrm{$n-t$ odd}.
    \end{cases}\]

To summarize, the constraint \eqref{eqn: bound on the partial sum symm}  is equivalent to the following.

\begin{enumerate}
\item If $\alpha\in (0,1/2]$, then
  \[\sum_{i=t}^n p_i\geq \begin{cases} (n-t)(2k-\frac{n-t}{2}-1), \quad &\textrm{if $n-t$ even},\\
    (n-t)(2k-\frac{n-t}{2}-1)-\frac{1}{2}, \quad &\textrm{if $n-t$ odd}.
    \end{cases}\]
\item If $\alpha\in (1/2,1]$, then
      \[\sum_{i=t}^n p_i\geq \begin{cases}(n-t)(2k-\frac{n-t}{2}), \quad &\textrm{$n-t$ even},\\
    (n-t)(2k-\frac{n-t}{2})+\frac{1}{2}, \quad &\textrm{$n-t$ odd}.
    \end{cases}\]
\end{enumerate}
Since $\cO_X$ is multiplicity-free, by \cite[Proposition 4.15]{multfreeinv} (see Page 29 and Lemma 4.13 for the notations) or \cite{abeasis},  the $G$-decomposition of $J_t^{(\ell)}$ contains representation $S_{\underline{p}}\C^n$ such that $\sum_{i=t}^n p_i \geq 2\ell$. As a consequence, this gives the desired primary decomposition.\end{proof}
\begin{remark}
The formula in Corollary \ref{cor: hodge ideal symmetric} matches up with known formulas for $I_k(\alpha D)$ and $\tilde{I}_k(\alpha D)$ in the literature. Set $n=2$, then $f=xy-z^2$ and the only nontrivial $J_t$ is $J_1$, the maximal ideal of the zero matrix. Up to a change of coordinates, this corresponds to the divisor $D=\textrm{div}(x^2+y^2+z^2)$ in $\C^3$. Denote by $\mathfrak{m}$ the ideal of the origin in $\C^3$. Then Corollary \ref{cor: hodge ideal symmetric} implies that
\[ I_k(\alpha D)=\tilde{I}_k(\alpha D)=\begin{cases}\mathfrak{m}^{(k-1)}=\mathfrak{m}^{k-1}, \quad &\textrm{if $\alpha\in (0,1/2]$}\\
\mathfrak{m}^{(k)}=\mathfrak{m}^{k}, \quad &\textrm{if $\alpha\in (1/2,1]$}.
\end{cases}\]

By \cite{SY23} (see Example 6.17 in the first arxiv version), one has  $\tilde{I}_k(\alpha D)=\mathfrak{m}^{\lceil k+\alpha-\frac{3}{2}\rceil}$. A simple computation shows that this matches up with formula above. Turning to Hodge ideals, it suffices to show
\[ I_k(\alpha D)=\mathfrak{m}^{\lceil k+\alpha-\frac{3}{2}\rceil}, \quad \textrm{whenever $0<\alpha\leq 1$}.\]
Note that $I_0(\alpha D)=\tilde{I}_0(\alpha D)=\cO_X=\mathfrak{m}^0$, because $\tilde{\alpha}_f=3/2>1$. Hence by \cite[Proposition D]{Zhang}, we have
\[ I_1(\alpha D)=\tilde{I}_1(\alpha D)=\mathfrak{m}^{\lceil 1+\alpha-\frac{3}{2}\rceil}.\]
So by \cite[Proposition D]{Zhang} again, we also know $I_2(\alpha D)=\tilde{I}_2(\alpha D)$. Arguing this inductively and using the fact that $\tilde{I}_k(\alpha D)=\mathfrak{m}^{\lceil k+\alpha-\frac{3}{2}\rceil}$, we conclude that 
\[ I_k(\alpha D)=\tilde{I}_k(\alpha D)=\mathfrak{m}^{\lceil k+\alpha-\frac{3}{2}\rceil}.\]
Alternatively, one can use \cite[Corollary B]{Zhang} to prove the claim. We leave the details for the interested readers.
\end{remark}

\subsection{Pfaffian}

Let $X = \bigwedge^2 (\C^n)^*$ be the space of skew-symmetric $n\times n$ matrices with the natural action of $G=\GL_n$, where $n$ is even. The semi-invariant function $f$ is the Pfaffian of a generic skew-symmetric matrix, with weight $\sigma = \det = (1^n)$ and $\deg f=n/2$. The $G$-decomposition of $\cM=(\cO_X)_f$ is given by (e.g. \cite[Proposition 2.3.8]{weymanbook})
\begin{equation}\label{eq:cauchy skew-symmetric}
\cM= \bigoplus_{\underline{p}} U_{\underline{p}}=\bigoplus_{\underline{p}\in \Z^n_{\mathrm{dom}},\, p_{2i-1} = p_{2i}} S_{\underline{p}} \C^n.
\end{equation}
As before, denote by $b_{\underline{p}}(s)$ the $b$-function associated to $\underline{p}$.

\begin{prop}\label{prop: bfunction iota determinant skewsymmetric}
Let $\underline{p}=(p_1,\dots, p_n)\in \Z^n_{\mathrm{dom}}$ with $p_{2i-1}=p_{2i}$ for all $1\leq i \leq n/2$. Then
\[ b_{\underline{p}}(s) = \prod_{i=1}^{n/2} (s+1+p_{2i}+n-2i).\]
\end{prop}

\begin{proof}
    The proof follows as in Proposition \ref{prop: bfunction iota determinant symmetric}, using the skew-symmetric version of the Capelli identity \cite{howeumeda}.
\end{proof}

\begin{corollary}\label{cor: Hodge ideal Pfaffian}
  Let $k\in \N$ and $\alpha\in (0,1]$, then there is a primary decomposition
  \[ \tilde{I}_k(\alpha D)=I_k(\alpha D)=\bigcap_{t=1}^{n/2} J_{2t}^{\left((\frac{n}{2}-t)(k-1)-(\frac{n}{2}-t)^2\right)},\]
  where $J_{2t}^{(d)}$ denotes the $d$-th symbolic power of $J_{2t}$, the ideal of functions vanishing to order $d$ along the subvariety of skew-symmetric matrices with rank $< 2t$.
\end{corollary}

\begin{proof}
    Since $\cO_X$ is multiplicity-free, we can compare both sides representation-theoretically. The result then reduces to a direct computation using \eqref{eqn: plambdaalpha in general}, Proposition \ref{prop: bfunction iota determinant skewsymmetric}, Corollary \ref{corollary: G decomposition of Hodge ideals} and \cite[Theorem 5.1]{Pfaffians}.
\end{proof}

\begin{remark}
    Observe that for $0<\alpha\ll 1$, one has $\tilde{I}_3(\alpha D)=J_{2(\frac{n}{2}-1)}=I_{D_{\mathrm{sing}}}\neq \cO_X$  is the first proper higher multiplier ideal, which is consistent with $\tilde{\alpha}_f=3$ (\cite[Lemma 3.10]{SY23}). Applying \cite[Theorem E]{MP16}, we obtain
    \[ I_k(D)\subseteq J_{2t}^{(q)}, \textrm{where $q=\min\left(\frac{n}{2}-t,(k+1)(\frac{n}{2}+1-t)-\binom{n-2t+2}{2}\right)$}.\]
     Our result shows that as $k$ becomes large, the optimal exponent is $(\frac{n}{2}-t)(k-1)-(\frac{n}{2}-t)^2$.
\end{remark}

As another corollary, we describe the Hodge filtrations on all the simple equivariant $\sD$-modules. Namely, let $\IC_{O_j}$ be the intersection cohomology $\sD$-module associated to $O_j$, the loci of rank $2j$ skew-symmetric matrices. 
By \cite[Theorem 5.7]{LW19} and Corollary \ref{thm: Gdecomposition of weight filtration}, it is easy to see that the socle filtration of $\cM=(\cO_X)_f$ must agree with the weight filtration. Since by definition $\gr^W_{\dim X+ \ell}\cM$ has weight $\dim X+\ell$ and $\IC_{O_{n/2-\ell}}$ has weight $\dim O_{n/2-\ell}=\dim X-\binom{2\ell}{2}$, we have
\[\gr^W_{\dim X+ \ell}\cM \cong \IC_{O_{n/2-\ell}}(-\ell^2), \quad \textrm{for all $0\leq \ell \leq n/2$},\]
where the Tate twist $(k)$ decreases the weight by $2k$.

\begin{corollary}\label{cor: simple for Phaffian} For $0 \leq j \leq n/2$, consider the set 
\[\Lambda^k_j = \{ \underline{p} \in \Z^n_{\mathrm{dom}} \mid p_{2i-1} = p_{2i}, \, p_{2j}\geq 2j-n, \, p_{2j+1}\leq 2j+1-n, \,\, - \sum_{i=j+1}^{n/2} p_{2i} \leq k \} \]
Then the $G$-decomposition of the Hodge filtration is given by
    \[F_k(\IC_{O_j}) = \bigoplus_{\underline{p}\in \Lambda^k_j} \, S_{\underline{p}}\, \C^n.\]
\end{corollary}
\begin{proof}
    By (\ref{eq:cauchy skew-symmetric}) and Corollaries \ref{thm: Gdecomposition of weight filtration} and \ref{cor: G decomposition of Hodge filtration}, we have a decomposition
    \[ F_k \gr_{\dim X+ \ell}^W (\mathcal{M}) =\bigoplus_{\substack{\underline{p}\in \Z^n_{\mathrm{dom}},\, p_{2i-1} = p_{2i}\\ \nu_{\underline{p}, 0}=\ell, \,\, \deg p_{\underline{p},0}(s) \leq k}} S_{\underline{p}} \C^n. \]
    Put $\ell = n/2-j$. By Proposition \ref{prop: bfunction iota determinant skewsymmetric}, we get $\nu_{\underline{p}, 0} = \ell$ if and only if $p_{2j}\geq 2j-n$ and $p_{2j+2}<2j+2-n$. Furthermore, we also get that $\deg p_{\underline{p}, 0}(s) \leq k $ if and only if $\sum_{i=j+1}^{n/2} (2i-n-p_{2i}-1) \leq k$. The claim now follows from
    \[ F_k(\IC_{O_j})=F_{k-(n/2-j)^2}\IC_{O_j}\left(-\left(\frac{n}{2}-j\right)^2\right)=F_{k-(n/2-j)^2}\gr^W_{\dim X+ \ell}\cM.\qedhere\]
\end{proof}

\subsection{Fundamental representation of $E_6$} 

Let $G = E_6 \times \C^*$, and let $X$ denote the 27-dimensional fundamental representation of $E_6$ (the simply-connected group of this type), with $\C^*$ acting by scalar multiplication. There exists a unique (up to scalar) semi-invariant $f$ with $\deg f = 3$, the Freudenthal cubic, whose $b$-function has roots $-1, -5, -9$ (see \cite{kimura}). Since there is no Capelli identity in this setting (cf. \cite[Section 11.13]{howeumeda}), we empoly a different approach to determine $b_{\lambda}(s)$.

Fix a Borel subgroup $B\subseteq G$. Up to scalar, there exist 3 irreducible $B$-semi-invariants $h_1, h_2, h_3 = f$ of weights $\lambda^1, \lambda^2, \lambda^3=\sigma$, with $\deg h_i = i$ (see \cite{howeumeda}). 

\begin{prop}\label{prop: b function E6}
    Let $\lambda = a_1 \lambda^1 + a_2 \lambda^2 + a_3 \lambda^3$ with $a_i \in \Z_{\geq 0}$. Then
    \[b_\lambda(s) = (s+a_3+1)(s+a_2+a_3+5)(s+a_1+a_2+a_3+9).\]
\end{prop}

\begin{proof}
We may assume that $a_3=0$. By Theorem \ref{thm: computation of blambda}, there exist $c_{ij} \in \Q_{\geq 0}$ such that
\[\partial f \cdot h_1^{a_1} h_2^{a_2} f^{s+1} = (s+c_{11}a_1+c_{12}a_2+1)(s+c_{21}a_1+c_{22}a_2+5)(s+c_{31}a_1+c_{32}a_2+9) \cdot h_1^{a_1} h_2^{a_2} f^{s}.\]
Substituting $s=-1$ yields $c_{11}=c_{12}=0$. It is easy to verify that $(\lambda^1)^* = \lambda^2 - \sigma$, and thus $\lambda^* = a_1 \lambda^2 + a_2 \lambda^1 - (a_1+a_2)\sigma$. Using Proposition \ref{prop:bfunsymmetry}, we deduce $c_{31} = c_{32}=1$, and $c_{21}+c_{22}=1$. Moreover, by Remark \ref{rem:cijnilp} we conclude $c_{2i} \in \Z_{\geq 0}$.

Assume, for contradiction, that $c_{21}=1$ and $c_{22}=0$. Then by Corollary \ref{thm: Gdecomposition of weight filtration}, we have $h_1 f^{-5} \in W_4 \cM$ and $h_2 f^{-5} \notin W_4 \cM$. Since $h_1\in \mathfrak{m}$, where $\mathfrak{m} \subseteq \cO_X$ is the homogeneous maximal ideal, it would follow that $\mathfrak{m} \cdot f^{-5}\subseteq W_4 \cM$, contradicting the fact that $h_2 \in \mathfrak{m}$. This contradiction completes the proof.
\end{proof}

\begin{corollary}\label{cor: Hodge ideal E6}
  Let $k\in \N$ and $\alpha\in (0,1]$, then there is a primary decomposition
  \[ \tilde{I}_k(\alpha D)=I_k(\alpha D)=I_1^{(2k-12)}\cap I_2^{(k-4)},\]
  where $I_1$ is the maximal ideal of the origin and $I_2$ is the ideal of the singular locus of $D$.
\end{corollary}
\begin{remark}
For $0<\alpha\ll 1$, we have $\tilde{I}_5(\alpha D)=I_2\neq \cO_X$ is the first proper ideal, which is consistent with $\tilde{\alpha}_f=5$ via \cite[Lemma 3.10]{SY23}. Applying \cite[Theorem E]{MP16} to $D_t=\mathrm{Zero}(I_t)$, we have
    \[ I_k(D)\subseteq I_t^{(q)}, \quad q=\begin{cases}\min (2,3k-24),&\textrm{$t=1$},\\
    \min(1,2k-8), &\textrm{$t=2$},\end{cases}\]
     given that $\textrm{mult}_{D_1}D=3,\textrm{mult}_{D_2}D=2,\mathrm{codim}D_1=27$ and $\mathrm{codim}D_2=10$. Our result shows that as $k$ becomes large, the optimal value of $q$ is $k-4$.
\end{remark}
\begin{proof}
We set up a correspondence between all $\lambda$ and $\Z^3_{\mathrm{dom}}$ via
\[ \lambda^1 = (1,0,0), \quad \lambda^2=(1,1,0), \quad \lambda^3=(1,1,1).\]
Given $\lambda=a_1\lambda^1+a_2\lambda^2+a_3\lambda^3$, this corresponds to $(b_1,b_2,b_3)\in \Z^3$ with 
\[b_1=a_1+a_2+a_3,\quad b_2=a_2+a_3, \quad b_3=a_3.\]
Using \eqref{eqn: plambdaalpha in general}, Proposition \ref{prop: b function E6}, and Corollary \ref{corollary: G decomposition of Hodge ideals}, it suffices to consider those $\lambda$ such that $\deg p_{\lambda,\alpha+k}(s)\leq k$, which is equivalent to
\[ \max(\lceil \alpha+k-b_3-1\rceil,0)+\max(\lceil \alpha+k-b_2-5\rceil,0)+\max(\lceil \alpha+k-b_1-9\rceil,0)\leq k.\]
%Since $b_1\geq b_2\geq 0$, we have $b_3+1\leq b_2+5\leq b_1+9$. 
Given $\alpha\in (0,1]$, it follows that this inequality reduces to
\[ b_3\geq 0, \, b_2+b_3\geq k-4,\, b_1+b_2+b_3\geq 2k-12.\]
The conclusion follows from \cite[Proposition 4.15]{multfreeinv} and that $\cO_X$ is multiplicity-free.\end{proof}

\subsection{Prehomogeneous spaces and holonomic functions}\label{sec:last} 

The results from Section \ref{sec:1.1} can be readily applied to other representations beyond multiplicity-free spaces, such as prehomogeneous spaces. Let $X$ be a prehomogeneous space with a semi-invariant function $f$, and set $\cM=(\cO_X)_f$. Then the isotypic components $\cM_\lambda$ are irreducible for all $\lambda = \sigma^k$, $k\in \Z$. Consequently, by Theorems \ref{thm:weightisot} and \ref{thm:hodgeisot}, the Bernstein--Sato polynomial of $f$ determines both the weight and Hodge level of the element $f^\alpha$ for any $\alpha \in \Q$. Considerable effort has been made toward computing these $b$-functions in various cases; see, for example \cite{kimura} for the irreducible prehomogeneous vector spaces.

\smallskip

The article \cite{holoprehom} obtains several explicit $b$-functions $b_\lambda(s)$ for $\cM \neq (\cO_X)_f$, that is, for localizations of simple $\sD$-modules corresponding to non-trivial local systems on the dense orbit. In this setting, elements of $\cM$ can be realized as $G$-finite holonomic (even algebraic) functions, and the results are obtained under the assumption that each $\cM_\lambda$ is irreducible (cf. \cite[Section 4]{holoprehom}).  Therefore, Theorem \ref{thm:weightisot}  applies directly to determine the weight levels of such functions in $\cM$. Building on this, we plan to pursue the full determination of the weight filtrations of these $\cM$ in future work.
\bibliographystyle{alpha}
\bibliography{references}

@article{nang,
author = {Nang, Philibert},
title = {On a class of holonomic {$\mathcal{D}$}-modules on {$M_n(\mathbb{C})$} related to the action of {$GL_n(\mathbb{C})\times GL_n(\mathbb{C})$}},
journal = {Adv. Math.},
volume = {218},
year = {2008},
pages = {635--648}
}

@article {LW19,
    AUTHOR = {L\H{o}rincz, Andr\'{a}s C. and Walther, Uli},
     TITLE = {On categories of equivariant {$\mathcal{D}$}-modules},
   JOURNAL = {Adv. Math.},
  FJOURNAL = {Advances in Mathematics},
    VOLUME = {351},
      YEAR = {2019},
     PAGES = {429--478},
      ISSN = {0001-8708},
   MRCLASS = {14F10 (14L30 14M27 16G20)},
}

@book{jacobson,
    AUTHOR = {Jacobson, Nathan},
     TITLE = {Basic algebra. {II}},
   EDITION = {Second},
      YEAR = {1989},
     PAGES = {xviii+686}
     }

@article {Saito88,
    AUTHOR = {Saito, Morihiko},
     TITLE = {Modules de {H}odge polarisables},
   JOURNAL = {Publ. Res. Inst. Math. Sci.},
  FJOURNAL = {Kyoto University. Research Institute for Mathematical
              Sciences. Publications},
    VOLUME = {24},
      YEAR = {1988},
    NUMBER = {6},
     PAGES = {849--995 (1989)},
      ISSN = {0034-5318},
   MRCLASS = {32C35 (14C30 32C38 32C42 32G99)},
  MRNUMBER = {1000123},
MRREVIEWER = {J. H. M. Steenbrink},
       DOI = {10.2977/prims/1195173930},
       URL = {https://doi.org/10.2977/prims/1195173930},
}

@article {Saito90,
    AUTHOR = {Saito, Morihiko},
     TITLE = {Mixed {H}odge modules},
   JOURNAL = {Publ. Res. Inst. Math. Sci.},
  FJOURNAL = {Kyoto University. Research Institute for Mathematical
              Sciences. Publications},
    VOLUME = {26},
      YEAR = {1990},
    NUMBER = {2},
     PAGES = {221--333},
      ISSN = {0034-5318},
   MRCLASS = {14D07 (14C30 32J25)},
}

@article{Saito16,
  title={Hodge ideals and microlocal {V}-filtration},
  author={Saito, Morihiko},
  journal={arXiv:1612.08667},
  year={2016},
}

@article{SY23,
url = {https://doi.org/10.1515/crelle-2025-0097},
title = {Higher multiplier ideals},
author = {Christian Schnell and Ruijie Yang},
pages = {153--197},
volume = {2026},
number = {832},
journal = {Journal für die reine und angewandte Mathematik (Crelles Journal)},
doi = {doi:10.1515/crelle-2025-0097},
year = {2026},
lastchecked = {2026-03-24}
}

@book {HTT,
    AUTHOR = {Hotta, Ryoshi and Takeuchi, Kiyoshi and Tanisaki, Toshiyuki},
     TITLE = {{$D$}-modules, perverse sheaves, and representation theory},
    SERIES = {Progress in Mathematics},
    VOLUME = {236},
      YEAR = {2008},
     PAGES = {xii+407},
      ISBN = {978-0-8176-4363-8},
   MRCLASS = {32C38 (14F05 14F10 17B10)},
  MRNUMBER = {2357361},
MRREVIEWER = {Corrado\ Marastoni},
       DOI = {10.1007/978-0-8176-4523-6},
       URL = {https://doi.org/10.1007/978-0-8176-4523-6},
}

@article{holoprehom,
    AUTHOR = {L{\H{o}}rincz, Andr\'{a}s Cristian},
     TITLE = {Holonomic functions and prehomogeneous spaces},
   JOURNAL = {Selecta Math. (N.S.)},
  FJOURNAL = {Selecta Mathematica. New Series},
    VOLUME = {29},
      YEAR = {2023},
    NUMBER = {5},
     PAGES = {Paper No. 69, 48}
}

@article{howeumeda,
    AUTHOR = {Howe, Roger and Umeda, Toru},
     TITLE = {The {C}apelli identity, the double commutant theorem, and
              multiplicity-free actions},
   JOURNAL = {Math. Ann.},
  FJOURNAL = {Mathematische Annalen},
    VOLUME = {290},
      YEAR = {1991},
    NUMBER = {3},
     PAGES = {565--619}
}

@article{class1,
    AUTHOR = {Benson, Chal and Ratcliff, Gail},
     TITLE = {A classification of multiplicity free actions},
   JOURNAL = {J. Algebra},
  FJOURNAL = {Journal of Algebra},
    VOLUME = {181},
      YEAR = {1996},
    NUMBER = {1},
     PAGES = {152--186}
}

@article{class2,
    AUTHOR = {Leahy, Andrew S.},
     TITLE = {A classification of multiplicity free representations},
   JOURNAL = {J. Lie Theory},
  FJOURNAL = {Journal of Lie Theory},
    VOLUME = {8},
      YEAR = {1998},
    NUMBER = {2},
     PAGES = {367--391}
}

@article{kimura,
    AUTHOR = {Kimura, Tatsuo},
     TITLE = {The {$b$}-functions and holonomy diagrams of irreducible regular prehomogeneous vector spaces},
   JOURNAL = {Nagoya Math. J.},
  FJOURNAL = {Nagoya Mathematical Journal},
    VOLUME = {85},
      YEAR = {1982},
     PAGES = {1--80}}

@incollection{multfreeactions,
    AUTHOR = {Benson, Chal and Ratcliff, Gail},
     TITLE = {On multiplicity free actions},
 BOOKTITLE = {Representations of real and {$p$}-adic groups},
    SERIES = {Lect. Notes Ser. Inst. Math. Sci. Natl. Univ. Singap.},
    VOLUME = {2},
     PAGES = {221--304},
 PUBLISHER = {Singapore Univ. Press, Singapore},
      YEAR = {2004},
      ISBN = {981-238-779-X},
   MRCLASS = {20G05 (22E46)},
  MRNUMBER = {2090872},
MRREVIEWER = {M\'aty\'as\ Domokos},
       DOI = {10.1142/9789812562500\_0006},
       URL = {https://doi.org/10.1142/9789812562500_0006},
}

@article{saki,
author = {M. Sato and T.Kimura},
title = {{A} classification of irreducible prehomogeneous vector spaces and their relative invariants},
journal={Nagoya Math. J.},
volume = {65},
year = {1977},
pages = {1--155}
}

@incollection {Sabbah,
    AUTHOR = {Sabbah, C.},
     TITLE = {{$D$}-modules et cycles \'{e}vanescents (d'apr\`es {B}.
              {M}algrange et {M}. {K}ashiwara)},
 BOOKTITLE = {G\'{e}om\'{e}trie alg\'{e}brique et applications, {III} ({L}a
              {R}\'{a}bida, 1984)},
    SERIES = {Travaux en Cours},
    VOLUME = {24},
     PAGES = {53--98},
 PUBLISHER = {Hermann, Paris},
      YEAR = {1987},
      ISBN = {2-7056-6029-1},
   MRCLASS = {32C40 (32B30 32C38 58G07)},
  MRNUMBER = {907935},
MRREVIEWER = {M.\ Sebastiani},
}

@article {PerlmanRaicu,
    AUTHOR = {Perlman, Michael and Raicu, Claudiu},
     TITLE = {Hodge ideals for the determinant hypersurface},
   JOURNAL = {Selecta Math. (N.S.)},
  FJOURNAL = {Selecta Mathematica. New Series},
    VOLUME = {27},
      YEAR = {2021},
    NUMBER = {1},
     PAGES = {Paper No. 1, 22},
      ISSN = {1022-1824,1420-9020},
   MRCLASS = {14M12 (13D45 14E15 14J17)},
  MRNUMBER = {4198526},
MRREVIEWER = {P.\ Schenzel},
       DOI = {10.1007/s00029-020-00616-z},
       URL = {https://doi.org/10.1007/s00029-020-00616-z},
}

@article{characters,
    AUTHOR = {Raicu, Claudiu},
     TITLE = {Characters of equivariant {${D}$}-modules on spaces of
              matrices},
   JOURNAL = {Compos. Math.},
  FJOURNAL = {Compositio Mathematica},
    VOLUME = {152},
      YEAR = {2016},
    NUMBER = {9},
     PAGES = {1935--1965}
}

@book{procesi,
    AUTHOR = {Procesi, Claudio},
     TITLE = {Lie groups: An approach through invariants and representations},
    SERIES = {Universitext},
 PUBLISHER = {Springer, New York},
      YEAR = {2007},
     PAGES = {xxiv+596}
}

@book{weymanbook,
	Address = {Cambridge},
	Author = {Weyman, Jerzy},
	Isbn = {0-521-62197-6},
	Mrclass = {13D02 (13C40 14L30 14M12 14M15 14M17 20G15)},
	Mrnumber = {1988690 (2004d:13020)},
	Mrreviewer = {Laurent Manivel},
	Pages = {xiv+371},
	Publisher = {Cambridge University Press},
	Series = {Cambridge Tracts in Mathematics},
	Title = {Cohomology of vector bundles and syzygies},
	Volume = {149},
	Year = {2003}
}

@article {MP16,
    AUTHOR = {Musta\c{t}\u{a}, Mircea and Popa, Mihnea},
     TITLE = {Hodge ideals},
   JOURNAL = {Mem. Amer. Math. Soc.},
  FJOURNAL = {Memoirs of the American Mathematical Society},
    VOLUME = {262},
      YEAR = {2019},
    NUMBER = {1268},
     PAGES = {v+80},
      ISSN = {0065-9266},
      ISBN = {978-1-4704-3781-7; 978-1-4704-5509-5},
   MRCLASS = {14D07 (14F17 14J17 32S25)},
}

@article {dCEP80,
    AUTHOR = {de Concini, C. and Eisenbud, David and Procesi, C.},
     TITLE = {Young diagrams and determinantal varieties},
   JOURNAL = {Invent. Math.},
  FJOURNAL = {Inventiones Mathematicae},
    VOLUME = {56},
      YEAR = {1980},
    NUMBER = {2},
     PAGES = {129--165},
      ISSN = {0020-9910,1432-1297},
   MRCLASS = {14M12 (14L30 15A72 20C30)},
  MRNUMBER = {558865},
MRREVIEWER = {Vladimir\ L.\ Popov},
       DOI = {10.1007/BF01392548},
       URL = {https://doi.org/10.1007/BF01392548},
}

@article {MPVfiltration,
    AUTHOR = {Musta\c{t}\u{a}, Mircea and Popa, Mihnea},
     TITLE = {Hodge ideals for {$\mathbb{Q}$}-divisors, {$V$}-filtration, and
              minimal exponent},
   JOURNAL = {Forum Math. Sigma},
  FJOURNAL = {Forum of Mathematics. Sigma},
    VOLUME = {8},
      YEAR = {2020},
     PAGES = {Paper No. e19, 41},
      ISSN = {2050-5094},
   MRCLASS = {14D07 (14F10 14J17 32S25)},
  MRNUMBER = {4089396},
MRREVIEWER = {Christian\ Schnell},
       DOI = {10.1017/fms.2020.18},
       URL = {https://doi.org/10.1017/fms.2020.18},
}

@article {Zhang,
    AUTHOR = {Zhang, Mingyi},
     TITLE = {Hodge filtration and {H}odge ideals for {$\mathbb{Q}$}-divisors
              with weighted homogeneous isolated singularities or convenient
              non-degenerate singularities},
   JOURNAL = {Asian J. Math.},
  FJOURNAL = {Asian Journal of Mathematics},
    VOLUME = {25},
      YEAR = {2021},
    NUMBER = {5},
     PAGES = {641--664},
      ISSN = {1093-6106,1945-0036},
   MRCLASS = {14F10 (14D07 14J17 32S25)},
  MRNUMBER = {4456022},
MRREVIEWER = {Luca\ Rizzi},
}

@article {Saito09,
    AUTHOR = {Saito, Morihiko},
     TITLE = {On the {H}odge filtration of {H}odge modules},
   JOURNAL = {Mosc. Math. J.},
  FJOURNAL = {Moscow Mathematical Journal},
    VOLUME = {9},
      YEAR = {2009},
    NUMBER = {1},
      ISSN = {1609-3321,1609-4514},
   MRCLASS = {32S40 (14C30 32C38 32S35)},
  MRNUMBER = {2567401},
MRREVIEWER = {C.\ A. M. Peters},
       DOI = {10.17323/1609-4514-2009-9-1-151-181},
       URL = {https://doi.org/10.17323/1609-4514-2009-9-1-151-181},
}

@article {Saito94,
    AUTHOR = {Saito, Morihiko},
     TITLE = {On microlocal {$b$}-function},
   JOURNAL = {Bull. Soc. Math. France},
  FJOURNAL = {Bulletin de la Soci\'{e}t\'{e} Math\'{e}matique de France},
    VOLUME = {122},
      YEAR = {1994},
    NUMBER = {2},
     PAGES = {163--184},
      ISSN = {0037-9484,2102-622X},
   MRCLASS = {32S40 (32C38 33E30 35A27)},
  MRNUMBER = {1273899},
MRREVIEWER = {Aleksandr\ G.\ Aleksandrov},
       URL = {http://www.numdam.org/item?id=BSMF_1994__122_2_163_0},
}

@article {JKSY,
    AUTHOR = {Jung, Seung-Jo and Kim, In-Kyun and Saito, Morihiko and Yoon,
              Youngho},
     TITLE = {Hodge ideals and spectrum of isolated hypersurface
              singularities},
   JOURNAL = {Ann. Inst. Fourier},
  FJOURNAL = {Universit\'{e} de Grenoble. Annales de l'Institut Fourier},
    VOLUME = {72},
      YEAR = {2022},
    NUMBER = {2},
     PAGES = {465--510},
      ISSN = {0373-0956,1777-5310},
   MRCLASS = {14J17 (14F10 32S25)},
  MRNUMBER = {4448602},
MRREVIEWER = {Jan\ Stevens},
       DOI = {10.5802/aif.3453},
       URL = {https://doi.org/10.5802/aif.3453},
}

@incollection {Malgrange,
    AUTHOR = {Malgrange, B.},
     TITLE = {Polyn\^{o}mes de {B}ernstein-{S}ato et cohomologie
              \'{e}vanescente},
 BOOKTITLE = {Analysis and topology on singular spaces, {II}, {III}
              ({L}uminy, 1981)},
    SERIES = {Ast\'{e}risque},
    VOLUME = {101-102},
     PAGES = {243--267},
 PUBLISHER = {Soc. Math. France, Paris},
      YEAR = {1983},
   MRCLASS = {58G07 (32B10 32C40)},
  MRNUMBER = {737934},
MRREVIEWER = {M.\ Sebastiani},
}

@article {Mochizuki,
    AUTHOR = {Mochizuki, Takuro},
     TITLE = {Wild harmonic bundles and wild pure twistor {$D$}-modules},
   JOURNAL = {Ast\'{e}risque},
  FJOURNAL = {Ast\'{e}risque},
    NUMBER = {340},
      YEAR = {2011},
     PAGES = {x+607},
      ISSN = {0303-1179,2492-5926},
      ISBN = {978-2-85629-332-4},
   MRCLASS = {14J60 (14F10 32L99)},
  MRNUMBER = {2919903},
MRREVIEWER = {Corrado\ Marastoni},
}

@book {MochizukitwistorDmodule,
    AUTHOR = {Mochizuki, Takuro},
     TITLE = {Mixed twistor {$\mathcal{D}$}-modules},
    SERIES = {Lecture Notes in Mathematics},
    VOLUME = {2125},
 PUBLISHER = {Springer, Cham},
      YEAR = {2015},
     PAGES = {xx+487},
      ISBN = {978-3-319-10087-6; 978-3-319-10088-3},
   MRCLASS = {53C28 (14F10 32C38 58A14)},
  MRNUMBER = {3381953},
       DOI = {10.1007/978-3-319-10088-3},
       URL = {https://doi.org/10.1007/978-3-319-10088-3},
}

@incollection {Kas83,
    AUTHOR = {Kashiwara, M.},
     TITLE = {Vanishing cycle sheaves and holonomic systems of differential
              equations},
 BOOKTITLE = {Algebraic geometry ({T}okyo/{K}yoto, 1982)},
    SERIES = {Lecture Notes in Math.},
    VOLUME = {1016},
     PAGES = {134--142},
 PUBLISHER = {Springer, Berlin},
      YEAR = {1983},
      ISBN = {3-540-12685-6},
   MRCLASS = {58G05 (14D05 32C38)},
  MRNUMBER = {726425},
MRREVIEWER = {P.\ Schapira},
       DOI = {10.1007/BFb0099962},
       URL = {https://doi.org/10.1007/BFb0099962},
}

@article{MHMproject,
  title={Mixed {H}odge {M}odule {P}roject},
  author={Sabbah, C. and Schnell, C.},
  journal={\url{https://perso.pages.math.cnrs.fr/users/claude.sabbah/MHMProject/mhm.pdf}},
  year={2025}
}

@incollection {Gyoja1996,
    AUTHOR = {Gyoja, Akihiko},
     TITLE = {Mixed {H}odge theory and prehomogeneous vector spaces},
      BOOKTITLE = {Research on prehomogeneous vector spaces (Kyoto,
              1996)},
   JOURNAL = {S\=urikaisekikenky\=usho K\B oky\=uroku},
  FJOURNAL = {S\=urikaisekikenky\=usho K\B oky\=uroku},
    NUMBER = {999},
      YEAR = {1997},
     PAGES = {116--132},
   MRCLASS = {14C30 (11G99)},
  MRNUMBER = {1622284},
}

@article {freedivisor,
    AUTHOR = {Casta\~no Dom\'inguez, Alberto and Narv\'aez Macarro, Luis and
              Sevenheck, Christian},
     TITLE = {Hodge ideals of free divisors},
   JOURNAL = {Selecta Math. (N.S.)},
  FJOURNAL = {Selecta Mathematica. New Series},
    VOLUME = {28},
      YEAR = {2022},
    NUMBER = {3},
     PAGES = {Paper No. 57, 62},
      ISSN = {1022-1824,1420-9020},
   MRCLASS = {32C38 (14F10 32S35 32S40)},
  MRNUMBER = {4412424},
       DOI = {10.1007/s00029-022-00767-1},
       URL = {https://doi.org/10.1007/s00029-022-00767-1},
}

@misc{BD24,
      title={{H}odge filtration and parametrically prime divisors}, 
      author={Daniel Bath and Henry Dakin},
      year={2024},
      eprint={2408.02601},
      archivePrefix={arXiv},
      primaryClass={math.AG},
      url={https://arxiv.org/abs/2408.02601}, 
}

@article {sato,
    AUTHOR = {Sato, Mikio},
     TITLE = {Theory of prehomogeneous vector spaces (algebraic part)},
   JOURNAL = {Nagoya Math. J.},
  FJOURNAL = {Nagoya Mathematical Journal},
    VOLUME = {120},
      YEAR = {1990},
     PAGES = {1--34}
}

@article{LY25,
      title={Filtrations of $\mathscr{D}$-modules and multiplicities of roots of {B}ernstein-{S}ato polynomials}, 
      author={Andr\'{a}s C. L\H{o}rincz and Ruijie Yang},
      year={2026},
      journal={In preparation},
}

@article{vanderbergh,
title = {Some generalities of {$G$}-equivariant quasi-coherent {$\mathscr{O}_X$} and {$\mathscr{D}_X$}-modules},
author = {Michel Van Den Bergh},
journal = {preprint},
note={\url{https://cantate.be/site/Geq.pdf}}
}

@article {BG,
    AUTHOR = {Braden, Tom and Grinberg, Mikhail},
     TITLE = {Perverse sheaves on rank stratifications},
   JOURNAL = {Duke Math. J.},
  FJOURNAL = {Duke Mathematical Journal},
    VOLUME = {96},
      YEAR = {1999},
    NUMBER = {2},
     PAGES = {317--362},
      ISSN = {0012-7094,1547-7398},
   MRCLASS = {14F43 (32S60)},
  MRNUMBER = {1666554},
MRREVIEWER = {Abdellah\ Mokrane},
       DOI = {10.1215/S0012-7094-99-09609-6},
       URL = {https://doi.org/10.1215/S0012-7094-99-09609-6},
}

@article {BS05,
    AUTHOR = {Budur, Nero and Saito, Morihiko},
     TITLE = {Multiplier ideals, {$V$}-filtration, and spectrum},
   JOURNAL = {J. Algebraic Geom.},
  FJOURNAL = {Journal of Algebraic Geometry},
    VOLUME = {14},
      YEAR = {2005},
    NUMBER = {2},
     PAGES = {269--282},
      ISSN = {1056-3911},
}

@article{MPbirational,
     author = {Musta\c{t}ǎ, Mircea and Popa, Mihnea},
     title = {Hodge ideals for {${\mathbb{Q}}$}-divisors: birational approach},
     journal = {Journal de l{\textquoteright}\'Ecole polytechnique {\textemdash} Math\'ematiques},
     pages = {283--328},
     publisher = {Ecole polytechnique},
     volume = {6},
     year = {2019},
     doi = {10.5802/jep.94},
     zbl = {07070281},
     mrnumber = {3959075},
     language = {en},
}

@misc{DLY,
      title={Archimedean zeta functions, singularities, and {H}odge theory}, 
      author={Dougal Davis and András C. Lőrincz and Ruijie Yang},
      year={2024},
      note={\href{https://arxiv.org/abs/2412.07849}{arXiv 2412.07849}}
}

@misc{hotta,
      title={Equivariant {$D$}-modules}, 
      author={Ryoshi Hotta},
      journal={\href{https://arxiv.org/abs/math/9805021}{arXiv:math/9805021}},
      year={1998}
}

@article{abeasis,
    AUTHOR = {Abeasis, Silvana},
     TITLE = {The {${\rm GL}(V)$}-invariant ideals in {$S(S\sp{2}V)$}},
   JOURNAL = {Rend. Mat. (6)},
  FJOURNAL = {Rendiconti di Matematica. Serie VI.},
    VOLUME = {13},
      YEAR = {1980},
    NUMBER = {2},
     PAGES = {235--262}
}

@article {DM05,
    AUTHOR = {de Cataldo, Mark Andrea A. and Migliorini, Luca},
     TITLE = {The {H}odge theory of algebraic maps},
   JOURNAL = {Ann. Sci. \'{E}cole Norm. Sup. (4)},
  FJOURNAL = {Annales Scientifiques de l'\'{E}cole Normale Sup\'{e}rieure. Quatri\`eme
              S\'{e}rie},
    VOLUME = {38},
      YEAR = {2005},
    NUMBER = {5},
     PAGES = {693--750},
      ISSN = {0012-9593},
   MRCLASS = {14D07 (32S60)},
  MRNUMBER = {2195257},
MRREVIEWER = {Christophe Mourougane},
}

@article{MM,
    author = {Maisonbe, Philippe and Mebkhout, Zoghman},
    title = {Le th{\'e}or{\`e}me de comparaison pour les cycles {\'e}vanescents},
    journal = {S\'emin. Congr.},
    volume = {8},
    year = {2004},
    pages = {311--389}}

@article {Pfaffians,
    AUTHOR = {Abeasis, S. and Del Fra, A.},
     TITLE = {Young diagrams and ideals of {P}faffians},
   JOURNAL = {Adv. in Math.},
  FJOURNAL = {Advances in Mathematics},
    VOLUME = {35},
      YEAR = {1980},
    NUMBER = {2},
     PAGES = {158--178},
      ISSN = {0001-8708},
   MRCLASS = {14L30 (15A72)},
  MRNUMBER = {560133},
       DOI = {10.1016/0001-8708(80)90046-8},
       URL = {https://doi.org/10.1016/0001-8708(80)90046-8},
}

@misc{dakin,
      title={Weight filtration and generating level}, 
      author={Henry Dakin},
      year={2025},
      note={\href{https://arxiv.org/abs/2503.14216}{arXiv 2503.14216}},
}

@misc{DY24,
      title={On the {H}odge and {V}-filtrations of mixed {H}odge modules}, 
      author={Dougal Davis and Ruijie Yang},
      year={2025},
      note={\href{https://arxiv.org/abs/2503.16619}{arXiv 2503.16619}, to appear in Sel. Math. New Ser.}, 
}

@article {multfreeinv,
    AUTHOR = {Ruitenburg, G. C. M.},
     TITLE = {Invariant ideals of polynomial algebras with multiplicity free
              group action},
   JOURNAL = {Compositio Math.},
  FJOURNAL = {Compositio Mathematica},
    VOLUME = {71},
      YEAR = {1989},
    NUMBER = {2},
     PAGES = {181--227},
      ISSN = {0010-437X,1570-5846},
   MRCLASS = {14L30 (14M12)},
  MRNUMBER = {1012639},
MRREVIEWER = {Michel\ Brion},
       URL = {http://www.numdam.org/item?id=CM_1989__71_2_181_0},
}

@article {LR20,
    AUTHOR = {L{\H o}rincz, Andr\'as C. and Raicu, Claudiu},
     TITLE = {Iterated local cohomology groups and {L}yubeznik numbers for
              determinantal rings},
   JOURNAL = {Algebra Number Theory},
  FJOURNAL = {Algebra \& Number Theory},
    VOLUME = {14},
      YEAR = {2020},
    NUMBER = {9},
     PAGES = {2533--2569},
      ISSN = {1937-0652,1944-7833},
   MRCLASS = {13D45 (13D07 14M12)},
  MRNUMBER = {4172715},
MRREVIEWER = {Roshan\ Tajarod},
       DOI = {10.2140/ant.2020.14.2533},
       URL = {https://doi.org/10.2140/ant.2020.14.2533},
}

@article {LR23,
    AUTHOR = {L{\H o}rincz, Andr\'as C. and Raicu, Claudiu},
     TITLE = {Borel-{M}oore homology of determinantal varieties},
   JOURNAL = {Algebr. Geom.},
  FJOURNAL = {Algebraic Geometry},
    VOLUME = {10},
      YEAR = {2023},
    NUMBER = {5},
     PAGES = {576--606},
      ISSN = {2313-1691,2214-2584},
   MRCLASS = {14M12 (14D07 32S35 55N33 55N35 57T15)},
  MRNUMBER = {4636284},
MRREVIEWER = {Zinovy\ Reichstein},
       DOI = {10.14231/ag-2023-020},
       URL = {https://doi.org/10.14231/ag-2023-020},
}

@article {Perlman,
    AUTHOR = {Perlman, Michael},
     TITLE = {Mixed {H}odge structure on local cohomology with support in
              determinantal varieties},
   JOURNAL = {Int. Math. Res. Not. IMRN},
  FJOURNAL = {International Mathematics Research Notices. IMRN},
      YEAR = {2024},
    NUMBER = {1},
     PAGES = {331--358},
      ISSN = {1073-7928,1687-0247},
   MRCLASS = {14C30 (14M12 32C38 32S35)},
  MRNUMBER = {4686653},
       DOI = {10.1093/imrn/rnad019},
       URL = {https://doi.org/10.1093/imrn/rnad019},
}

@article {Kac,
    AUTHOR = {Kac, V. G.},
     TITLE = {Some remarks on nilpotent orbits},
   JOURNAL = {J. Algebra},
  FJOURNAL = {Journal of Algebra},
    VOLUME = {64},
      YEAR = {1980},
    NUMBER = {1},
     PAGES = {190--213},
      ISSN = {0021-8693},
   MRCLASS = {17B30 (17B70 20G05)},
  MRNUMBER = {575790},
MRREVIEWER = {Dragomir\ \v Z.\ Djokovi\'c},
       DOI = {10.1016/0021-8693(80)90141-6},
       URL = {https://doi-org.colorado.idm.oclc.org/10.1016/0021-8693(80)90141-6},
}

@article {Brion,
    AUTHOR = {Brion, M.},
     TITLE = {Repr\'esentations exceptionnelles des groupes semi-simples},
   JOURNAL = {Ann. Sci. \'Ecole Norm. Sup. (4)},
  FJOURNAL = {Annales Scientifiques de l'\'Ecole Normale Sup\'erieure.
              Quatri\`eme S\'erie},
    VOLUME = {18},
      YEAR = {1985},
    NUMBER = {2},
     PAGES = {345--387},
      ISSN = {0012-9593},
   MRCLASS = {14L30 (14M17 20G05 22E45)},
  MRNUMBER = {816368},
MRREVIEWER = {H.\ H.\ Andersen},
       URL = {http://www.numdam.org/item?id=ASENS_1985_4_18_2_345_0},
}

\vspace{\baselineskip}

\footnotesize{
\textsc{Department of Mathematics, University of Oklahoma, 660 Parrington Oval, Norman, OK 73019, United States} \\
\indent \textit{E-mail address:} \href{mailto:lorincz@ou.edu}{lorincz@ou.edu}

\vspace{\baselineskip}

\textsc{Department of Mathematics, University of Kansas, 1450 Jayhawk Blvd, Lawrence, KS 66045, United States} \\
\indent \textit{E-mail address:} \href{mailto:ruijie.yang@ku.edu}{ruijie.yang@ku.edu} 
}

\end{document}